\documentclass[12pt]{extarticle}
\usepackage{amsmath, amsthm, amssymb, color}
\usepackage[colorlinks=true,linkcolor=blue,urlcolor=blue]{hyperref}
\usepackage{graphicx}
\usepackage{caption}
\usepackage{mathtools}
\usepackage{enumerate}
\usepackage[all]{xypic}
\usepackage{verbatim}
\usepackage{tikz}
\usetikzlibrary{matrix}
\usetikzlibrary{arrows}
\usepackage{algorithm}
\usepackage[noend]{algpseudocode}
\usepackage{caption}
\usepackage{subcaption}
\usepackage[normalem]{ulem}
\usepackage{tikz-cd}
\oddsidemargin \evensidemargin
\usepackage{makecell}
\usepackage{array}

\theoremstyle{theorem}
\newtheorem{theorem}{Theorem}
\newtheorem{proposition}[theorem]{Proposition}
\newtheorem{lemma}[theorem]{Lemma}

\theoremstyle{definition}

\newtheorem{remark}[theorem]{Remark}

\newtheorem{example}[theorem]{Example}

\numberwithin{theorem}{section}

\newcommand{\PP}{\mathbb{P}}
\newcommand{\RR}{\mathbb{R}}
\newcommand{\QQ}{\mathbb{Q}}
\newcommand{\CC}{\mathbb{C} }

\newcommand{\NN}{\mathbb{N}}
\newcommand{\KK}{\mathbb{K}}

\title{\bf Linear PDE with Constant Coefficients}
\author{Rida Ait El Manssour,  Marc H\"ark\"onen and Bernd Sturmfels}

\date{}

\begin{document}
\maketitle

\begin{abstract}
\noindent 
We discuss practical methods for computing the space of solutions to an arbitrary
homogeneous linear system of partial differential equations with constant coefficients.
These rest on the Fundamental Principle of Ehrenpreis--Palamodov from
the 1960s. We develop this further using
recent advances in
computational commutative algebra.
 \end{abstract}

 \section{Introduction}
 \label{sec1}

Our calculus class taught us how to solve
ordinary differential equations (ODE) of the form
\begin{equation}
\label{eq:ODE}
  c_0 \phi + c_1 \phi ' + c_2 \phi '' + \cdots + c_m \phi^{(m)} \,\, = \,\, 0 . 
\end{equation}  
Here we seek functions $\phi = \phi(z)$ in one unknown $z$.
The ODE is linear of order $m$, it has constant coefficients $c_i \in \CC$, and it
is homogeneous, meaning that the right hand side is zero.
The set of all solutions  is a vector space of dimension $m$. A basis consists of
$m$ functions
\begin{equation}
\label{eq:ODEsol} \qquad \qquad \phi(z) \,\, = \,\, z^a \cdot {\rm exp}(u_i z) .
\end{equation}
Here  $u_i$ is a complex zero with multiplicity larger than $a \in \NN$ of the characteristic polynomial
\begin{equation} \label{eq:charpoly}
p(x) \,\, = \,\,  c_0 + c_1 x + c_2 x^2 + \cdots + c_m x^m .
\end{equation}
Thus solving the ODE (\ref{eq:ODE}) means finding all the zeros of
 (\ref{eq:charpoly}) and their multiplicities.

We next turn to a partial differential equation (PDE) for functions $\phi: \RR^2 \rightarrow \RR$
that is familiar from the undergraduate curriculum, namely  the one-dimensional wave equation
\begin{equation} \qquad
\label{eq:wave1} \phi_{tt}(z,t) \,\,= \,\,c^2 \,\phi_{zz}(z,t), \qquad {\rm where} \,\,\,c \in \RR \backslash \{0\} .
\end{equation}
D'Alembert found in 1747 that the general solution is the superposition of traveling waves,
\begin{equation}
\label{eq:wave2}
\phi(z,t) \,\, = \,\, f(z+ct) \,+\, g(z-ct) , 
\end{equation}
where $f$ and $g$ are twice differentiable functions in one variable.
For the special parameter value $c=0$, the PDE (\ref{eq:wave1}) becomes $\phi_{tt} = 0$, and
the general solution has still two summands
\begin{equation}
\label{eq:wave3}
 \phi(z,t) \,\, = \,\, f(z) \,+\, t \cdot h'(z). 
 \end{equation}
We get this from (\ref{eq:wave2}) by replacing
$ g(z-ct)$ with $ \frac{1}{2c}(h(z{+}ct) - h(z{-}ct)) $ and taking the limit $c \rightarrow 0$.
Here, the role of the characteristic polynomial (\ref{eq:charpoly}) is played by the quadratic form
\begin{equation}
\label{eq:qc}
 q_c(u,v) \quad = \quad v^2 - c^2 u^2 \,\, = \,\, (v-cu)(v+cu). 
\end{equation}
The solutions (\ref{eq:wave2}) and (\ref{eq:wave3}) mirror
the algebraic geometry of the conic $\{q_c=0\}$ for any $c \in \RR$.

\smallskip

Our third example is a system of  three PDE for unknown functions
$$ \psi \,:\, \RR^4 \,\rightarrow\, \CC^2 \,, \,\, (x,y,z,w) \,\mapsto\,  \bigl(\alpha(x,y,z,w),\beta(x,y,z,w) \bigr) . $$
Namely, we consider the following linear PDE with constant coefficients:
\begin{equation}
\label{eq:pdesys1} \alpha_{xx} + \beta_{xy} \, = \, \alpha_{yz} + \beta_{zz} \, = \, \alpha_{xxz} + \beta_{xyw}\,= \,0. 
\end{equation}
The general solution to this system has nine summands,
labeled $a,b, \ldots,h$ and $(\tilde \alpha,\tilde \beta)$:
\begin{equation}
\label{eq:pdesys2}
\! \begin{matrix}
\!\! \alpha &\! = \! &\!\!  a_z(y,\!z,\!w)- b_y(x,\!y)+c(y,\!w)+xd(y,\!w)+xg(z,\!w)-xyh_z(z,\!w)+  \tilde \alpha(x,y,z,w), \\
\!\!  \beta &\! = \! &  -a_y(y,z,w)\,+\,b_x(x,y)\,+\,e(x,w)\,+\,zf(x,w) \,+\,xh(z,w)\,+ \tilde \beta(x,y,z,w) . 
\end{matrix}
\end{equation}
Here, $a$ is any function in three variables, $b,c,d,e,f,g,h$ are functions in two variables,
and $\tilde \psi = (\tilde \alpha, \tilde \beta)$ is any function $\RR^4 \rightarrow \CC^2$ that satisfies the 
following four linear PDE of first order:
\begin{equation}
\label{eq:pdesys3} \tilde \alpha_{x} +  \tilde \beta_{y} \, = \,
\tilde \alpha_{y} + \tilde \beta_{z} \, = \, \tilde \alpha_{z} - \tilde \alpha_{w}\,= \,
\tilde \beta_{z} - \tilde \beta_{w} \, = \, 0 . 
\end{equation}
We note that all solutions to (\ref{eq:pdesys3}) also satisfy (\ref{eq:pdesys1}), and
they admit the integral representation
\begin{equation}
\label{eq:pdesys4} 
\tilde \alpha = 
\! \int \!\! t({\rm exp}(s^2x+sty+t^2(z{+}w))) d\mu(s,t)  \,, \,\,\,
\tilde \beta = 
- \!\!\int \! \! s({\rm exp}(s^2x+sty+t^2(z{+}w))) d\mu(s,t),
\end{equation}
where $\mu$ is a measure on $\CC^2$. All functions in (\ref{eq:pdesys2})
are assumed to be suitably differentiable.

\smallskip

Our aim is to present
 methods for  solving arbitrary systems of 
homogeneous linear PDE
with constant coefficients. The input is a system like (\ref{eq:ODE}),
(\ref{eq:wave1}), (\ref{eq:pdesys1}) or (\ref{eq:pdesys3}). We seek to compute
the corresponding output (\ref{eq:ODEsol}),
(\ref{eq:wave2}), (\ref{eq:pdesys2}) or (\ref{eq:pdesys4}) respectively.
We present techniques that are based
on the Fundamental Principle of Ehrenpreis and Palamodov,
as discussed in the classical books \cite{BJORK, EHRENPREIS,
HORMANDER, PALAMODOV}. We utilize the 
theory of differential primary decomposition \cite{CS}.
While deriving (\ref{eq:wave2})  from (\ref{eq:wave1})  is easy by hand,
getting from (\ref{eq:pdesys1}) to (\ref{eq:pdesys2})  requires a computer.

This article is primarily expository. One goal is to explain the findings
 in \cite{CCHKL, CC, CHKL, CPS, CS}, such as the
 differential primary decompositions of minimal size, from the viewpoint of analysis and PDE. In addition to these
recent advances, our development rests on  a considerable
body of earlier work. The articles \cite{DAMIANO, oberst95, oberst99} are especially important.
However, there are also some new contributions in the present article,
mostly in Sections \ref{sec4}, \ref{sec5} and \ref{sec6}.
We describe the first universally applicable algorithm for computing
Noetherian operators.

This presentation
is organized as follows.
Section \ref{sec2} explains how linear PDE are represented by polynomial modules.
The Fundamental Principle (Theorem \ref{thm:Palamodov_Ehrenpreis})
is  illustrated with concrete examples.
In Section \ref{sec3} we examine the support of a module, and 
how it governs exponential solutions (Proposition \ref{prop:46})
 and polynomial solutions (Proposition \ref{cor:aredense}).
 Theorem~\ref{thm:finitedim} characterizes
  PDE whose solution space is finite-dimensional.
   Section \ref{sec4} features the theory of
  differential primary decomposition
  \cite{CC, CS}. 
   Theorem \ref{thm:DPD_to_integrals}
  shows how this theory yields the  integral representations promised
 by Ehrenpreis--Palamodov. This result appeared implicitly in the analysis literature,
 but the present algebraic form is new.
 It~is the foundation of our algorithm for computing a minimal set of
 Noetherian multipliers. This is
 presented in Section \ref{sec5}, along with its
 implementation in the 
 command {\tt solvePDE} in
 {\tt Macaulay2} \cite{MAC2}.

 The concepts of schemes and coherent sheaves
 are central to modern algebraic geometry.
 In Section \ref{sec6} we argue that linear PDE are
 an excellent tool for understanding these concepts,
 and for computing their behaviors in families. Hilbert schemes and Quot schemes 
 make an appearance   along the lines of   \cite{CC, CPS}.
Section \ref{sec7} is devoted to
directions for further study and research in the subject~area of this paper.
It also features more examples and applications.

\section{PDE and Polynomials}
\label{sec2}

Our point of departure is the observation that 
homogeneous linear partial differential equations 
with constant coefficients are the same as vectors of polynomials.
The entries of the vectors are elements in the polynomial ring
$R = K[\partial_{1}, \partial_{2}, \ldots,\partial_{n}]$,
where $K$ is a subfield of the complex numbers $\CC$. In all our examples
we use the field $K = \QQ$ of rational numbers. This has the virtue of being
amenable to exact symbolic computation, e.g.~in {\tt Macaulay2}~\cite{MAC2}.

For instance, in (\ref{eq:ODE}), we have $n=1$. Writing
$\partial = \frac{\partial}{\partial z}$ for the generator of $R$, our ODE is given by
one polynomial  $p(\partial) = c_0 + c_1 \partial + \cdots + c_m \partial^m$,
where $c_0,c_1,\ldots,c_m \in K$.
For $n \geq 2$, we write ${\bf z} = (z_1,\ldots,z_n)$ for the unknowns in the 
functions we seek, and the partial derivatives that act on these functions are
$\partial_i = \partial_{z_i} = \frac{\partial}{\partial z_i}$. With this notation,
the wave equation in (\ref{eq:wave1}) corresponds to the polynomial
$q_c(\partial) = \partial_2^2 - c^2 \partial_1^2 = (\partial_2 - c \partial_1)(\partial_2 + c \partial_1)$ with
 $n=2$.
Finally, the PDE in (\ref{eq:pdesys1}) has $n=4$ and is encoded in three polynomial vectors
\begin{equation}
\label{eq:threevectors}
\begin{pmatrix} \partial_1^2 \\ \partial_1 \partial_2 \end{pmatrix} \, , \quad 
\begin{pmatrix} \partial_2 \partial_3 \\ \partial_3^2  \end{pmatrix} \quad {\rm and} \quad
\begin{pmatrix} \partial_1^2 \partial_3 \\ \partial_1 \partial_2 \partial_4 \end{pmatrix} .
\end{equation}
The system  (\ref{eq:pdesys1}) corresponds to the submodule of $R^2$ that is generated by
these three vectors.

We shall study PDE that describe vector-valued functions
from $n$-space to $k$-space. 
To this end, we need to specify a space $\mathcal{F}$ of
sufficiently differentiable functions
such that $\mathcal{F}^k$ contains our solutions. The 
scalar-valued functions in $\mathcal{F}$ are either real-valued functions
$\psi:\Omega \rightarrow \RR$ or complex-valued functions
$\psi:\Omega \rightarrow \CC$, where
$\Omega$ is a suitable subset of $\RR^n$ or $\CC^n$.
Later we will be more specific about the choice of $\mathcal{F}$.
One requirement is that
the space $\mathcal{F}^k$ should contain the {\em exponential functions}
\begin{equation}
\label{eq:polexpfun} q({\bf z}) \cdot {\rm exp}( {\bf u}^t {\bf z}) \,\, = \,\,
q(z_1,\ldots,z_n) \cdot {\rm exp}( u_1 z_1 + \cdots + u_n z_n).  
\end{equation}
Here ${\bf u} \in \CC^n$ and $q$ is any vector of length $k$ whose entries are polynomials in $n$ unknowns.

\begin{remark}[$k=1$]
A differential operator $p(\partial) $ in $R$ annihilates the 
function ${\rm exp}({\bf u}^t {\bf z})$  if and only if $p({\bf u}) = 0$.
This is the content of \cite[Lemma 3.25]{INLA}. See also
Lemma \ref{cor:module_exponential_solution}.
If $p(\partial)$ annihilates a function $ q({\bf z}) \cdot {\rm exp}( {\bf u}^t {\bf z})$,
where $q$ is a polynomial of positive degree, then 
${\bf u}$ is a point of higher multiplicity on the hypersurface $\{ p = 0 \}$.
In the case $n=1$, when $p$ is the characteristic polynomial (\ref{eq:charpoly}),
we have a solution basis of exponential functions (\ref{eq:ODEsol}).
\end{remark}

Another requirement for the space  $\mathcal{F}$  is that it is closed under  differentiation.
In other words, if $\phi = \phi(z_1,\ldots,z_n)$ lies in $\mathcal{F}$ then so does
$\partial_i \bullet \phi = \frac{\partial \phi}{\partial z_i}$ for $i=1,2,\ldots,n$.
The elements of $\mathcal{F}^k$ are vector-valued functions $\psi = \psi({\bf z})$.
Their coordinates $\psi_i$ are scalar-valued functions
in $\mathcal{F}$.
All in all, $\mathcal{F}$ should be large,
in the sense that it furnishes enough solutions.
Formulated algebraically, we want
$\mathcal{F}$ to be an {\em injective $R$-module}
 \cite{Lomadze}. A more precise desideratum,
 formulated by Oberst  \cite{oberst90,oberst95, oberst96}, is that
 $\mathcal{F}$ should be an {\em injective cogenerator}.

Examples of injective cogenerators include the ring $\CC[[z_1,\dotsc,z_n]]$ of formal power series, the space $C^\infty(\RR^n)$ of smooth
complex-valued functions over $\RR^n$, or more generally, the space $\mathcal{D}'(\RR^n)$ of complex-valued distributions on $\RR^n$. If $\Omega $ is any open convex domain in $\RR^n$ then we
can also  take $\mathcal{F}$ to be $C^\infty(\Omega)$ or $\mathcal{D}'(\Omega)$.
In this paper we focus on algebraic methods.
Analytic difficulties are mostly swept under the rug.

Our PDE are elements in the free $R$-module $R^k$, that is, they are column vectors
of length $k$ whose entries are polynomials in $\partial = (\partial_1,\ldots,\partial_n)$.
Such a vector acts on $\mathcal{F}^k$ by coordinate-wise application of the differential
operator and then adding up the results in $\mathcal{F}$.
In this manner, each element in $R^k$ defines an $R$-linear map 
$\mathcal{F}^k \rightarrow \mathcal{F}$.
For instance, the third vector in (\ref{eq:threevectors}) is an element in $R^2$ that
acts on functions $\psi : \RR^4 \rightarrow \CC^2$ in $\mathcal{F}^2$ as follows:
\begin{equation}
\label{eq:doubleaction}
\begin{pmatrix} \partial_1^2 \partial_3 \\ \partial_1 \partial_2 \partial_4 \end{pmatrix}  \bullet
(\psi_1 ({\bf z} ) ,\psi_2({\bf z})) \quad = \quad
\frac{\partial^3 \psi_1}{ \partial z_1^2 \partial z_3}\, + \,
\frac{\partial^3 \psi_2}{ \partial z_1 \partial z_2 \partial z_4}. 
\end{equation}
The right hand side is a scalar-valued function $\RR^4 \rightarrow \CC$,
that is, it is an element of $\mathcal{F}$.

Our systems of PDE are submodules $M$ of the free module $R^k$.
By Hilbert's Basis Theorem, every module $M$ is finitely generated, so 
we can write $M = {\rm image}_R(A)$, where $A$ is a $k \times l$ matrix
with entries in $R$. Each column of $A$ is a generator of $M$ and it 
defines a differential operator that maps $\mathcal{F}^k$ to $\mathcal{F}$.
The solution space to the PDE given by $M$ equals
\begin{equation} 
\label{eq:solM}
{\rm Sol}(M) \,\, := \,\, \bigl\{ \,\psi \in \mathcal{F}^k \,:\, m \bullet \psi = 0 \,\,\, \hbox{for all} \, \,\, m \in M \,\bigr\}.
\end{equation}
It suffices to take the operators $m$ from a generating set of $M$, such as the $l$ columns of $A$.
The case $k=1$ is of special interest, since we  often consider PDE for scalar-valued functions.
In that case, the submodule is an ideal in the polynomial ring $R$ and we denote this by $I$.
The solution space ${\rm Sol}(I)$ of the ideal $I \subseteq R$ is the set of functions
$\phi$ in $\mathcal{F}$ such that $p(\partial) \bullet \phi = 0 $ for all $p \in I$.
Thus ideals are instances of modules, with their own~notation.

The solution spaces ${\rm Sol}(M)$ and ${\rm Sol}(I)$ are $\CC$-vector spaces and $R$-modules.
Indeed, any $\CC$-linear combination of solutions is again a solution.
The $R$-module action 
means applying the same differential operator $p(\partial)$ to each coordinate, 
which leads to another
vector in~$\mathcal{F}^k$. 
This action takes solutions to solutions because the ring of differential operators with constant coefficients $R = \mathbb{C}[\partial_1,\dotsc,\partial_n]$ is commutative.

The purpose of this paper is to present practical methods for the following task:
\begin{equation}
\label{eq:task1} \begin{matrix}
\hbox{\em Given a $k \times l$ matrix $A$ with entries in $R = K[\partial_1,\ldots,\partial_n]$, compute a good} \\
\hbox{\em representation for the solution space ${\rm Sol}(M)$ of the module $M = {\rm image}_R(A)$.}
\end{matrix}
\end{equation}
If $k=1$ then we consider the ideal $I$ generated by the entries of $A$ 
and we compute ${\rm Sol}(I)$.

This raises the question of what a ``good representation'' means.
The formulas  in (\ref{eq:ODEsol}),
(\ref{eq:wave2}), (\ref{eq:pdesys2}) and (\ref{eq:pdesys4}) 
are definitely good. They guide us to what is desirable.
Our general answer stems from the following important result
 at the crossroads of analysis and algebra.
It  involves two sets of unknowns
${\bf z} = (z_1,\ldots,z_n)$ and ${\bf x}= (x_1,\ldots,x_n)$.
Here ${\bf x}$ gives coordinates
on certain  irreducible varieties $V_i$
in $\CC^n$ that are parameter spaces for solutions.
Our solutions $\psi$ are functions in ${\bf z}$.
We take $\mathcal{F} = C^\infty(\Omega)$ where $\Omega \subset \RR^n$ is open,
convex, and bounded.

 \begin{theorem}[Ehrenpreis--Palamodov Fundamental Principle]
 	\label{thm:Palamodov_Ehrenpreis}
Consider a module $M \subseteq R^k$, representing linear PDE for a function 
	$\psi: \Omega \rightarrow \CC^k$.	
  There exist irreducible varieties $V_1,\ldots,V_s$ in $\CC^n$ and
  finitely many vectors  $B_{ij}$ of polynomials in $2n$ unknowns
  $({\bf x},{\bf z})$,
  all independent of the set $\Omega$,
   such that any solution $\psi \in \mathcal{F}$
  admits an integral representation
\begin{equation}
\label{eq:anysolution}
 	\psi(\mathbf{z}) \,\,\,= \,\,\, \sum_{i=1}^s \sum_{j=1}^{m_i} \,\int_{V_i} \!\! B_{ij} \! \left(\mathbf{x},\mathbf{z}\right) 
	\exp\left( \mathbf{x}^t \,\mathbf{z} \right) d\mu_{ij} (\mathbf{x}).
\end{equation}
Here $m_i$ is a certain invariant of $(M,V_i)$ and
each $\mu_{ij}$ is a bounded measure supported on the  variety~$V_i $.
\end{theorem}

Theorem \ref{thm:Palamodov_Ehrenpreis} appears in different forms
in the books by Bj\"ork  \cite[Theorem 8.1.3]{BJORK},
Ehrenpreis \cite{EHRENPREIS},
H\"ormander \cite[Section 7.7]{HORMANDER} and
Palamodov \cite{PALAMODOV}. Other references
with different emphases include \cite{BP, Lomadze, oberst95}.
For a perspective from commutative algebra see~\cite{CPS, CS}.

In the next sections we will study the ingredients in
Theorem \ref{thm:Palamodov_Ehrenpreis}. Given the module~$M$, we 
compute each {\em associated variety} $V_i$,
the {\em arithmetic length} $m_i$ of $M$ along $V_i$,
and the {\em Noetherian multipliers} $B_{i,1},B_{i,2},\ldots,B_{i,m_i}$.
We shall see that not all $n$ of the unknowns $z_1,\ldots,z_n$ appear
in the polynomials $B_{i,j}$ but only a subset of ${\rm codim}(V_i)$ of them.

\smallskip

The most basic example is the  ODE in (\ref{eq:ODE}), with $l=n=k=1$.
Here $V_i = \{u_i\}$ is the $i$th root of the polynomial (\ref{eq:charpoly}),
which has multiplicity $m_i$, and
$B_{i,j} = z^{j-1}$.
 The measure $\mu_{ij}$ is a scaled Dirac measure on $u_i$, 
 so the integrals in (\ref{eq:anysolution}) are 
 multiples of the basis functions~(\ref{eq:ODEsol}).

In light of Theorem \ref{thm:Palamodov_Ehrenpreis}, we now refine
our computational task in (\ref{eq:task1}) as follows:
\begin{equation}
\label{eq:task2} \begin{matrix}
\hbox{\em Given a $k \times l$ matrix $A$ with entries in $R = K[\partial_1,\ldots,\partial_n]$, 
compute  the varieties $V_i$}
 \\ \hbox{\em and the Noetherian multipliers
$B_{ij}({\bf x},{\bf z})$.
This encodes ${\rm Sol}(M)$ for $M = {\rm image}_R(A)$.}
\end{matrix}
\end{equation}
In our introductory examples we gave formulas for the general solution, namely
(\ref{eq:wave2}) and (\ref{eq:pdesys2}).
We claim that such formulas can be read off from the
 integrals in (\ref{eq:anysolution}). For instance,
 for the wave equation (\ref{eq:wave1}),
 we have $s=2$,   $B_{1,1} = B_{1,2} = 1$,
 and (\ref{eq:wave2}) is obtained by integrating
 ${\rm exp}( {\bf x}^t {\bf z})$ against measures
$ d\mu_{i1}({\bf x})$ on two lines $V_1$ and $V_2$ in $ \CC^2$.
For the system (\ref{eq:pdesys1}), we find $s=6$, with
$m_1=m_2=m_3=1$ and $m_4=m_5=m_6=2$, and
the nine integrals in (\ref{eq:anysolution})
translate into (\ref{eq:pdesys2}). We shall explain such a translation
in full detail for two other examples.

\begin{example}[$n=3,k=1,l=2$] \label{ex:312}
The ideal $I = \langle \partial_1^2 - \partial_2 \partial_3, \partial_3^2 \rangle $ represents the PDE
\begin{equation}
\label{eq:anton} \frac{\partial^2 \phi}{\partial z_1^2} \,=\, \frac{\partial^2 \phi} {\partial z_2 \partial z_3}
\qquad {\rm and} \qquad \frac{\partial^2 \phi}{\partial z_3^2} \, = \, 0  
\end{equation}
for a scalar-valued function $\phi = \phi(z_1,z_2,z_3)$.
 This is \cite[Example 4.2]{CHKL}. A {\tt Macaulay2}
 computation as in Section \ref{sec5} shows that
 $s=1, m_1 =4$. It reveals the  Noetherian multipliers
$$ B_1 = 1\,,\, \, B_2 = z_1\,,\, \, B_3 = z_1^2 x_2 + 2 z_3\,,\, \, B_4 = z_1^3 x_2 + 6 z_1 z_3 . $$
 Arbitrary functions $\,f(z_2) = \int \! {\rm exp}( t z_2 ) dt\,$
  are obtained by integrating against suitable measure on the line
 $V_1= \{ (0,t,0) \,:\,t \in \CC \} \subset \CC^3$.
Their derivatives are found by differentiating under the integral sign, namely
$\,f'(z_2) =  \int t \cdot {\rm exp}( t z_2 )dt$.
 Consider four functions $a,b,c,d$,
each specified by a different measure. Thus the sum of the four integrals in (\ref{eq:anysolution})~evaluates to
\begin{equation}
\label{eq:niceintrep}
 \phi({\bf z}) \,\, = \,\, a(z_2) \, + \, z_1 b(z_2) \, + \,( z_1^2 c'(z_2) + 2 z_3 c(z_2) ) \,   
+ \,( z_1^3 d'(z_2) + 6 z_1z_3 d(z_2)). 
\end{equation}
According to Ehrenpreis--Palamodov, this sum is the general solution of the PDE (\ref{eq:anton}).
\end{example}

Our final example uses concepts from
 primary decomposition, to be reviewed in~Section~\ref{sec3}.

\begin{example}[$n=4,k=2,l=3$] \label{ex:mod1a}
Let $M \subset R^4$  be the module generated by the columns~of 
 \begin{equation}
 \label{eq:exfour}
 A \,\, = \,\, \begin{bmatrix} \,
 \partial_{1} \partial_{3} &  \partial_{1} \partial_{2} &  \partial_{1}^2 \partial_{2} \, \,\\ \,
 \partial_{1}^2 & \partial_{2}^2 & \partial_{1}^2 \partial_{4} \,\,
 \end{bmatrix}. 
 \end{equation}
 Computing ${\rm Sol}(M)$ means solving
  $\, \frac{\partial^2 \psi_1}{\partial z_1 \partial z_3}  
 +  \frac{\partial^2 \psi_2}{\partial z_1^2}   \, = \,
      \frac{\partial^2 \psi_1}{\partial z_1 \partial z_2} +   \frac{\partial^2 \psi_2}{\partial z_2^2} 
      \,= \,
      \frac{\partial^3 \psi_1}{\partial z_1^2 \partial z_2} +
      \frac{\partial^3 \psi_2}{\partial z_1^2 \partial z_4} \,=\,0$.
Two  solutions are $\psi({\bf z}) = \bigl(\phi(z_2,z_3,z_4) , 0\bigr)$
   and  $\,\psi({\bf z}) =  {\rm exp}(s^2 t z_1 + s t^2 z_2 + s^3 z_3 + t^3 z_4) \cdot 
   \bigl( t , -s \bigr)$.

We apply Theorem \ref{thm:Palamodov_Ehrenpreis} to 
derive the general
solution to~(\ref{eq:exfour}).
The module $M$ has six associated primes, namely
$P_1 = \langle \partial_{1} \rangle$,
$P_2 = \langle \partial_{2}, \partial_{4} \rangle $,
$P_3 = \langle \partial_{2}, \partial_{3} \rangle $,
$P_4 = \langle \partial_{1}, \partial_{3} \rangle $,
$P_5 = \langle \partial_{1}, \partial_{2} \rangle $, and
$P_6 = \langle \partial_{1}^2 - \partial_2 \partial_3, \partial_1 \partial_2 - \partial_3 \partial_4,
\partial_2^2 - \partial_1 \partial_4 \rangle $.
Four of them are minimal and two are
embedded. We find that
$m_1 = m_2 = m_3 = m_4 = m_6 = 1$ and $m_5 = 4$.
A minimal primary decomposition 
\begin{equation}
\label{eq:MPD}
 M \,= \, M_1 \,\cap \,M_2 \,\cap \, M_3 
 \,\cap \, M_4  \,\cap \, M_5  \,\cap \, M_6 
 \end{equation}
 is given by the following primary submodules of $R^4$, each of which contains $M$:
 \begin{small}
$$ 
 M_1 = {\rm im}_R \begin{bmatrix}
 \partial_1 & 0  \\
    0           & 1 
 \end{bmatrix},\,\,\quad
  M_2 = {\rm im}_R \begin{bmatrix}
  \partial_2 & \partial_4 &       0         &         0      & \partial_3 \\
       0         &         0      &  \partial_2 & \partial_4 & \partial_1 
\end{bmatrix},\,\,\qquad
 M_3 = {\rm im}_R \begin{bmatrix}
 \partial_2 & \partial_3 & 0  \\
     0          &    0           & 1 
 \end{bmatrix},\,\,
 $$
 $$
  M_4 = {\rm im}_R \begin{bmatrix}
 \partial_3^5 & \partial_1 &  0 & 0 \\
        0          &  \partial_2 & \partial_1 & \partial_3 
 \end{bmatrix},\,\,\,\,
  M_5 = {\rm im}_R \begin{bmatrix}
  \partial_1 & \partial_2^5 & 0 & 0 \\
  0 & 0 & \partial_1^2 & \partial_2^2 
 \end{bmatrix},\,\,
  M_6 = {\rm im}_R \begin{bmatrix}
\partial_1  & \partial_2  & \partial_3  \\
\partial_2 & \partial_4  & \partial_1
\end{bmatrix}.
 $$
\end{small}
 The number of Noetherian multipliers $B_{ij}$ is
 $\sum_{i=1}^6 m_i = 9$. We choose them to be
  \begin{small} $$
 B_{1,1} {=} \begin{bmatrix}  1  \\ 0  \end{bmatrix}  ,\,
  B_{2,1} {=} \begin{bmatrix}   \phantom{-}x_1 \\   -x_3 \end{bmatrix} ,\,
   B_{3,1} {=} \begin{bmatrix}  1  \\   0 \end{bmatrix}  , \,
 B_{4,1} = \begin{bmatrix}   x_2 z_1 \\ -1  \end{bmatrix} , \,
  B_{5,i} = \begin{bmatrix}   0 \\ z_1 z_2  \end{bmatrix}  \!,
    \begin{bmatrix}   0 \\ z_1  \end{bmatrix} \! ,
      \begin{bmatrix}   0 \\ z_2  \end{bmatrix} \! , 
  \begin{bmatrix}   0 \\ 1  \end{bmatrix} , \,
    B_{6,1} = \begin{bmatrix}   \phantom{-}x_4 \\ -x_2  \end{bmatrix}.
$$    \end{small}
These nine vectors describe all solutions to our PDE. For instance,
$B_{3,1}$ gives the solutions
\begin{small} $  \begin{bmatrix} \alpha(z_1,z_4) \\ 0 \end{bmatrix}$, \end{small} 
and $B_{5,1}$ 
gives the solutions
\begin{small} $  \begin{bmatrix} 0 \\ z_1 z_2 \beta(z_3,z_4) \end{bmatrix}$,  \end{small}
where $\alpha,\beta$ are bivariate functions.
Furthermore
$B_{1,1}$ and $B_{6,1}$ encode
the two families of solutions mentioned after (\ref{eq:exfour}).

For the latter, we note that $V_6 = V(P_6) $ is the surface in $\CC^4$
with parametric representation $\,(x_1,x_2,x_3,x_4) =  (s^2 t, st^2, s^3,t^3)\,$
for $s,t \in \CC$. This surface is the cone over the twisted cubic curve, 
in the same notation as in \cite[Section 1]{CPS}.
The kernel under the integral in  (\ref{eq:anysolution}) equals
$$  \begin{bmatrix}   \phantom{-}x_4 \\ -x_2  \end{bmatrix}  
{\rm exp}\bigl(x_1 z_1 + x_2 z_2 + x_3 z_3 + x_4 z_4\bigr)  \quad = \quad 
t^2   \begin{bmatrix}   \phantom{-} t \\ - s  \end{bmatrix}  
 {\rm exp}\bigl( s^2t z_1 + st^2 z_2 + s^3 z_3 + t^3 z_4 \bigr).
$$
This is a solution to $M_6$, and hence to $M$, for any values of $s$ and $t$.
Integrating the left hand side over ${\bf x} \in V_6$ amounts to 
integrating the right hand side over $(s,t) \in \CC^2$. Any such integral
is also a solution. Ehrenpreis--Palamodov tells us that these are all the solutions.
\end{example}

\section{Modules and Varieties}
\label{sec3}

Our aim is to offer practical tools for solving PDE. The input is a 
$k \times l$ matrix $A$ with entries in $R = K[\partial_1,\ldots,\partial_n]$,
and $M = {\rm image}_R(A)$ is the corresponding submodule of~$R^k = \bigoplus_{j=1}^k Re_j$.
The output is the description of ${\rm Sol}(M)$ sought in (\ref{eq:task2}).
That description is  unique up to basis change,
in the sense of \cite[Remark 3.8]{CS}, by the discussion in Section \ref{sec4}.
Our method is implemented in a 
 {\tt Macaulay2} command, called {\tt solvePDE} and to 
be described in Section~\ref{sec5}.

We now explain the ingredients of Theorem \ref{thm:Palamodov_Ehrenpreis}
coming from commutative algebra
(cf.~\cite{Eisenbud}).
 For a vector $m \in R^k$, the quotient $(M:m)$ is
the ideal $\{f \in R : fm \in M\}$. A prime ideal $P_i \subseteq R$ is 
{\em associated to} $M$ if there exists $m \in R^k$ such that
$(M:m) = P_i$. Since $R$ is Noetherian, the list of
associated primes of $M$ is finite, say $P_1,\ldots,P_s$.
If $s=1$ then the module $M$ is called {\em primary} or {\em $P_1$-primary}.
A {\em primary decomposition} of $M$ is a list of primary submodules
$M_1,\ldots,M_s \subseteq R^k$ 
where $M_i$ is $P_i$-primary and
$\, M  =  M_1 \cap M_2 \cap \cdots \cap M_s $.

Primary decomposition is a standard topic in commutative algebra.
It is usually presented for ideals $(k=1)$,
as in \cite[Chapter 3]{INLA}. The case of modules is
analogous.
The latest version of {\tt Macaulay2} has an implementation of
primary decomposition for modules, as described in \cite[Section~2]{CC}.
Given $M$, 
the primary module $M_i$ is not unique if $P_i$
is an embedded prime.


The contribution of the primary module $M_i$ to $M$
is quantified by a positive integer $m_i$, called 
  the  arithmetic length of $M$ along $P_i$.
To define this, we consider the localization
$(R_{P_i})^k/M_{P_i}$. This is a module over
the local ring $R_{P_i}$. The {\em arithmetic length} 
is the length of the largest submodule of finite length in
$(R_{P_i})^k/M_{P_i}$; in symbols,
$m_i = {\rm length} \bigl( H^0_{P_i} ((R_{P_i})^k/ M_{P_i})\bigr)$.
The sum $m_1 + \cdots + m_s$ is an invariant of the module $M$,
denoted ${\rm amult}(M)$, and known as the {\em arithmetic multiplicity} of $M$.
These numbers can be computed in {\tt Macaulay2} as in
\cite[Remark 5.1]{CS}. We return to these invariants in Theorem~\ref{thm:DPD}.

To make the connection to Theorem \ref{thm:Palamodov_Ehrenpreis},
we set $V_i = V(P_i)$ for $i=1,2,\ldots,s$. Thus, $V_i$ is the
irreducible variety in $\CC^n$ defined by the prime ideal $P_i$
in $R = K[\partial_1,\ldots,\partial_n]$. The integer
$m_i$ is an invariant of the pair $(M,V_i)$: it 
measures the thickness
of the module $M$ along~$V_i$.

By taking the union of the irreducible varieties $V_i$ we obtain the variety
$$ V(M) \quad := \quad V_1 \,\cup\, V_2 \,\cup \,\cdots \,\cup\, V_s 
\quad \subset \,\, \CC^n.
$$
Algebraists refer to $V(M)$ as the {\em support} of $M$, while analysts call it 
the {\em characteristic variety} of $M$.
The support is generally reducible, with $\leq s$ irreducible components.
For instance, the module $M$ in Example \ref{ex:mod1a} has
six associated primes, and an explicit primary decomposition was given in
(\ref{eq:MPD}). However, the support $V(M)$ has only four irreducible components in $\CC^4$,
namely  one hyperplane, two $2$-dimensional planes, and one nonlinear surface
(twisted cubic).

\begin{remark} If $k=1$ and $M=I$ then the support $V(M)$ coincides with the
variety $V(I)$ attached as usual to an ideal $I$, namely 
the common zero set in $\CC^n$ of all polynomials in $I$.
\end{remark}

The relationship between modules and ideals mirrors
the relationship between PDE
for vector-valued functions and related PDE
for scalar-valued functions.
To pursue this a bit further,
we now define two ideals that are naturally associated with a given module $M
\subseteq R^k$.

The first ideal is the {\em annihilator}
of the quotient module $R^k/M = {\rm coker}_R(A)$, which is 
$$ I \,\, := \,\,{\rm Ann}_R(R^k/M) \,\,  = \,\,
\big\{ \,f \in R \,:\, f m \in M \,\,\hbox{for all} \,\, m \in R^k  \bigr\} . $$
The second is the zeroth {\em Fitting ideal} of $R^k/M$,
which is the ideal in $R$ generated by the $k \times k$ minors of the 
presentation matrix $A$. It is independent of the choice of $A$, and we write
$$ J \,\, := \,\, {\rm Fitt}_0(R^k/M)  \,\,= \,\, \bigl\langle
\hbox{$\,k \times k$ subdeterminants of $A$}\,\bigr\rangle. $$
We are interested in the affine varieties in $\CC^n$ defined by these ideals.
They are denoted by $V(I)$ and $V(J)$ respectively. The following is
a standard result in commutative algebra.

\begin{proposition} \label{prop:MIJ}
The three varieties above are equal for every submodule $M$ of $R^k$, that is,
\begin{equation}
\label{eq:VM}
 V(M) \, = \, V(I) \, = \,V(J) \,\, \subseteq \,\, \CC^n. 
 \end{equation}
\end{proposition}

\begin{proof}
This follows from \cite[Proposition 20.7]{Eisenbud}.
\end{proof}

 \begin{remark}
It can happen that ${\rm rank}(A) < k$, 
for instance when  $k > l$.  In that case, 
$I = J = \{ 0 \}$ and 
 $V(M) = \CC^n$.
Geometrically, the module $M$ furnishes a coherent sheaf
that is supported on the entire space $\CC^n$.
For instance, let $k=n=2,l=1$ and $A = \binom{\phantom{-}\partial_1}{-\partial_2}$.
The PDE asks for pairs $(\psi_1,\psi_2)$ such that
$\partial \psi_1 /\partial z_1 =  \partial \psi_2 /\partial z_2 $.
We see that $\operatorname{Sol}(M)$ consists of all pairs
$\bigl( \partial \alpha/ \partial z_2\,,\partial \alpha/ \partial z_1 \big)$, where $\alpha 
= \alpha(z_1,z_2)$
runs over functions in two variables. In general, the left kernel of $A$
furnishes differential operators for creating solutions to $M$.
 \end{remark}

The following example shows that \eqref{eq:VM}
is not true at the level of schemes (cf.~Section~\ref{sec6}).

\begin{example}[$n=k=3,l=5$]
Let $R = \CC[\partial_1,\partial_2,\partial_3]$ and $M $ the submodule of $R^3$ given~by 
$$ A \,\, = \,\,
\begin{pmatrix}
\,\partial_1 &  0 &  0 & 0 &  0 \,\,\\
\, 0 & \partial_1^2 & \partial_2 & 0 & 0 \,\,\\
\, 0 & 0 & 0 & \partial_1 & \,\partial_3 \,\, 
\end{pmatrix} . $$
We find $I = \langle \partial_1^2, \partial_1 \partial_2 \rangle \, \supset \,
J = \langle \partial_1^4, \partial_1^3 \partial_3,
\partial_1^2 \partial_2 , \partial_1 \partial_2 \partial_3 \rangle$.
The sets of associated primes are
$$  \begin{matrix}
& \operatorname{Ass}(I) & = &  \bigl\{ \langle \partial_1 \rangle, \langle \partial_1, \partial_2 \rangle \bigr\} 
&\qquad & {\rm with} \,\,{\rm amult}(I) =2 \\ 
  \subset & \operatorname{Ass}(M) & = & 
 \bigl\{ \langle \partial_1 \rangle, \langle \partial_1, \partial_2 \rangle ,
 \langle \partial_1, \partial_3 \rangle  \bigr\} & \qquad & {\rm with} \,\,{\rm amult}(M) = 4 
 \\
 \subset & \operatorname{Ass}(J) & = &
 \bigl\{ \langle \partial_1 \rangle, \langle \partial_1, \partial_2 \rangle ,
 \langle \partial_1, \partial_3 \rangle,
 \langle \partial_1, \partial_2, \partial_3 \rangle
   \bigr\} & \qquad & {\rm with} \,\,{\rm amult}(J) = 5
 \end{matrix}
 $$
  The support $V(M)$ is a plane in $3$-space, on which
 $I$ and $J$ define different scheme structures.
 Our module $M$ defines a coherent sheaf on that plane
 that lives between these two schemes.
We consider the PDE in each of the three cases, 
we compute the Noetherian multipliers,
and from this we derive the general solution.
To begin with, functions in $  {\rm Sol}(J)$ have the form
$$ \alpha(z_2,z_3) \,+\, z_1 \beta(z_3) \,+\, z_1^2 \gamma(z_3) \,+\,
z_1 \delta(z_2) \,+\, c \cdot z_1^3 . $$ 
The first two terms give functions in the subspace ${\rm Sol}(I)$.
Elements in ${\rm Sol}(M)$ are~vectors
 $$ \bigl( \,\rho(z_2,z_3) \,,\, \sigma(z_3) + z_1 \tau(z_3)\,,\, \omega(z_2) \,\bigr) .
 $$
 These represent all functions $\CC^3 \rightarrow \CC^3$ that satisfy
 the five PDE given by the matrix $A$.
 \end{example}
 
 \begin{remark}
The quotient $R/I$ embeds naturally into the direct sum of
$k$ copies of $R^k/M$, via $1 \mapsto e_j$.
This implies ${\rm Ass}(I) \subseteq {\rm Ass}(M)$.
 It would be worthwhile to  understand
how the differential primary decompositions of $I,J$ and $M$ are related,
and to study implications for the solution spaces
${\rm Sol}(I)$, ${\rm Sol}(J)$ and ${\rm Sol}(M)$. What relationships
hold between these?
 \end{remark}

\begin{lemma}\label{cor:module_exponential_solution}
Fix a $k \times l$ matrix $A(\partial)$ and its module $M \subseteq R^k$
as above.  A point ${\bf u}\in \CC^n$ lies in $V(M)$ if and only if
there exist constants  $c_1,\dotsc,c_k \in \CC$, not all zero, such that
  \begin{align} \label{eq:constant_exponential_solution}
    \begin{pmatrix}
      c_1 \\ \vdots \\ c_k
    \end{pmatrix} \exp(u_1z_1 + \dotsb + u_nz_n) \,\in\, \operatorname{Sol}(M).
  \end{align}
More precisely, \eqref{eq:constant_exponential_solution} holds 
if and only if  $\,  (c_1,\dotsc,c_k) \cdot A({\bf u}) = 0 $.
  \end{lemma}

  \begin{proof}
  Let $a_{ij}(\partial)$ denote the entries of the matrix $A(\partial)$. Then
     \eqref{eq:constant_exponential_solution} holds if and only if
  \begin{align*}
    \sum_{i = 1}^k a_{ij}(\partial) \bullet (c_i \exp(u_1z_1+\dotsb + u_nz_n)) \,=\, 0 
    \quad \text{ for all } j = 1,\dotsc, l.
  \end{align*}
  This is equivalent to
  \begin{align*}
    \sum_{i=1}^k c_i \,a_{ij}({\bf u}) \exp(u_1z_1+\dotsb + u_nz_n) \,=\, 0 \quad \text{ for all } j = 1,\dotsc,l.
  \end{align*}
This condition holds if and only if 
$\,  (c_1,\dotsc,c_k) \cdot A({\bf u}) $ is the zero vector in $\CC^l$.
We conclude that, for any given ${\bf u} \in \CC^n$,
the previous condition is satisfied for some $c \in \CC^k \backslash \{0\}$
if and only if ${\rm rank}(A({\bf u}))  < k$ if and only if ${\bf u} \in V(M) = V(I)$.
Here we use Proposition \ref{prop:MIJ}.
\end{proof}

Here is an alternative way to interpret the characteristic variety of a system of PDE:

\begin{proposition} \label{prop:46}
  The solution space $\,\operatorname{Sol}(M)$ contains an
  exponential solution 
  $ q({\bf z}) \cdot {\rm exp}( {\bf u}^t {\bf z}) $
    if and only if $\,{\bf u} \in V(M)$. Here $q $ is some
    vector of $k$ polynomials in $n$ unknowns,  as in (\ref{eq:polexpfun}).
\end{proposition}

\begin{proof}
  One direction is clear from Lemma \ref{cor:module_exponential_solution}. Next, suppose $  q(\mathbf{z}) \exp({\bf u}^t {\bf z}) \in \operatorname{Sol}(M)$.
The partial derivative of this function with respect to any unknown $z_i$
is also in ${\rm Sol}(M)$. Hence
  \begin{align*}
    \partial_i \bullet (q({\bf z}) \exp({\bf u}^t {\bf z}))
 \,    =\, (\partial_i \bullet q({\bf z})) \exp({\bf u}^t {\bf z})\,+\,
  u_i q({\bf z}) \exp({\bf u}^t {\bf z}) \,\in\, \operatorname{Sol}(M)
  \quad \hbox{for $i=1,\ldots,n$.}
  \end{align*}
  Hence the exponential function
$  (\partial_i \bullet q({\bf z})) \exp({\bf u}^t {\bf z})$
is in ${\rm Sol}(M)$.
Since the degree of $\partial_i \bullet q({\bf z})$ is less than that of $q({\bf z})$,
  we can find a sequence $D = \partial_{i_1} \partial_{i_2} \dotsb \partial_{i_s}$ such that $D \bullet q$ is a nonzero constant vector and $(D \bullet q) \exp({\bf u}^t {\bf z})
   \in \operatorname{Sol}(M)$.
  Lemma \ref{cor:module_exponential_solution} now implies that ${\bf u} \in V(M)$.
\end{proof}

The solution space ${\rm Sol}(M)$ to a submodule
$M \subseteq R^k$  is a vector space over $\CC$.
It is infinite-dimensional
 whenever $V(M)$ is a variety of positive dimension.
 This follows from Lemma~\ref{cor:module_exponential_solution}
 because there are infinitely many points
${\bf u}$ in $V(M)$. 
However, if
$V(M)$ is a finite subset of $\CC^n$, then
 ${\rm Sol}(M)$ is 
finite-dimensional. 
This is the content of the next theorem.

\begin{theorem} \label{thm:finitedim}
Consider a module $M \subseteq R^k$, viewed as a system of linear PDE.
  Its solution space $\operatorname{Sol}(M)$ is  finite-dimensional  over
   $\mathbb{C}$  if and only if $V(M)$ has dimension $0$.
  In this case,  ${\rm dim}_\CC \operatorname{Sol}(M) = {\rm dim}_K(R^k/M) =
   {\rm amult}(M)$. There is a basis of ${\rm Sol}(M)$ given by vectors
   $ \,q({\bf z})\, {\rm exp}({\bf u}^t {\bf z})$, where
   ${\bf u} \in V(M)$ and $q({\bf z})$ runs over a finite set of polynomial vectors,
   whose cardinality is the length of $M$ along the maximal ideal
      $\langle x_1 - u_1,\ldots,x_n-u_n \rangle$.
        There exist polynomial solutions if and only if  $\mathfrak{m} = \langle x_1,\dotsc,x_n \rangle$ is an associated prime of $M$. 
  The polynomial solutions are found by solving
     the PDE given by the $\mathfrak{m}$-primary component of~$M$.
\end{theorem}

\begin{proof}
This is the main result in Oberst's article \cite{oberst96},
proved in the setting of  injective cogenerators $\mathcal{F}$.
The same statement for $\mathcal{F} = C^\infty(\Omega)$
 appears in \cite[Ch.~8, Theorem 7.1]{BJORK}.
 The scalar case $(k=1)$ is  found in \cite[Theorem 3.27]{INLA}.
 The proof given there uses solutions in the power series ring,
 which is an injective cogenerator,
   and it generalizes to modules.
 \end{proof}

By a {\em polynomial solution} we mean a vector $q({\bf z})$ whose
coordinates are polynomials.
The $\mathfrak{m}$-primary component in Theorem \ref{thm:finitedim}
is computed by a double saturation step. When $M=I$ is an  ideal then
 this double saturation is $I:(I:\mathfrak{m}^\infty)$, as seen in
 \cite[Theorem 3.27]{INLA}. For submodules $M$ of $R^k$ with $k \geq 2$, 
 we would compute
$ M : {\rm Ann}(R^k / (M : \mathfrak{m}^\infty) ) $.
  The inner colon
$(M:\mathfrak{m}^\infty)$ is the intersection of all primary components of $M$
whose variety $V_i$ does not contain the origin $0$. It is computed as
$(M:f) = \{m \in R^k: fm \in M \}$, where $f$ is a random homogeneous
polynomial of large degree.
The outer colon is the module $(M:g)$, where $g$ is a general polynomial in
the ideal $\operatorname{Ann}(R^k/(M:f))$.
See also \cite[Proposition 2.2]{CC}.

It is an interesting problem to identify polynomial solutions
when $V(M)$ is no longer finite, and to decide whether these are dense
in the infinite-dimensional space of all solutions. Here ``dense'' refers to the topology on
$\mathcal{F}$ used by Lomadze in \cite{Lomadze2019}.
The following result gives an algebraic characterization of the closure
in ${\rm Sol}(M)$  of the subspace of polynomial solutions.

\begin{proposition} \label{cor:aredense}
The polynomial solutions are dense in ${\rm Sol}(M)$ if and only if 
the origin $0$ lies in every associated variety $V_i$ of the module $M$.
If this fails then the topological closure of the space of
polynomial solutions $q({\bf z})$ to $M$ is
the solution space of $M : \operatorname{Ann}(R^k/(M : \mathfrak{m}^\infty))$.
\end{proposition}

\begin{proof}
This proposition is our reinterpretation of 
Lomadze's result in \cite[Theorem 3.1]{Lomadze2019}.
\end{proof}

The result gives rise to algebraic algorithms for answering analytic questions about
a system of PDE.
The property in the first sentence can be decided by
running the primary decomposition algorithm in \cite{CC}.
For the second sentence, we need to compute a double saturation as above.
This can be carried out in \texttt{Macaulay2} as well.

\section{Differential Primary Decomposition}
\label{sec4}

We now shift gears and pass to a setting that is dual to the one we have seen so far.
Namely, we discuss differential primary decompositions \cite{CC, CS}.
That duality is  subtle and can be confusing at first sight.
To mitigate this, we introduce new notation. We set
$x_i = \partial_i = \partial_{z_i}$ for $i=1,\ldots,n$. Thus $R$ is now
the polynomial ring $K[x_1,\ldots,x_n]$. This is the
notation we are used to from  algebra courses (such as~\cite{INLA}).
We write $\partial_{x_1},\ldots,\partial_{x_n}$ for the
differential operators corresponding to $x_1,\ldots,x_n$.
Later on, we also identify $z_i  = \partial_{x_i}$, 
and we think of the unknowns ${\bf x}$ and ${\bf z}$ in
the multipliers $B_i({\bf x},{\bf z})$ as dual in the sense
of the Fourier transform.

The ring of differential operators on the polynomial ring $R$ is the Weyl algebra
$$ D_n \,=\, R \langle \partial_{x_1},\ldots,\partial_{x_n} \rangle \,=\,
K \langle x_1,\ldots,x_n, \partial_{x_1},\ldots,\partial_{x_n} \rangle . $$
The $2n$ generators commute, except for the $n$
relations $\partial_{x_i} x_i - x_i \partial_{x_i} = 1$, which expresses
  the Product Rule from Calculus. Elements in the Weyl algebra $D_n$
are linear differential operators with polynomial coefficients. 
We write $\delta \bullet p$ for the result of applying $\delta \in D_n$ to a
polynomial $p = p({\bf x})$ in $ R$. For instance, $x_i \bullet p = x_i \,p$
and $\partial_{x_i} \bullet p = \partial p/\partial x_i$.
 Let  $D_n^k$ denote
the  $k$-tuples of differential operators in $D_n$. 
These operate on the free module $R^k$ as follows:
$$ D_n^k \times R^k \rightarrow R \,:\, (\delta_1,\ldots,\delta_k) \bullet (p_1,\ldots,p_k)
\,= \, \sum_{j=1}^k \delta_j \bullet p_j. $$

Fix a submodule $M$ of $R^k$ and let $P_1,\ldots,P_s$ be its associated primes,
as in Section \ref{sec3}.
A {\em differential primary decomposition} of $M$ is a list
$\mathcal{A}_1,\ldots,\mathcal{A}_s$ of finite subsets of $D_n^k$ such that
\begin{equation}
\label{eq:modulemembership}
M \,\,\, = \,\,\, \bigl\{ \,m \in R^k \,: \, \delta \bullet m \in P_i 
\,\,\,\,\hbox{for all}\,\,\, \delta \in \mathcal{A}_i \,\,\hbox{and}\,\, i=1,2, \ldots,s \,\bigr\}. 
\end{equation}
This is a membership test for the module $M$
using differential~operators. This test is geometric since
the polynomial $\delta \bullet m $ lies in $P_i $ if and only if 
it vanishes on the variety $V_i=V(P_i)$.

\begin{theorem}\label{thm:diff_prim_dec_existence}
  Every submodule $M$ of $R^k$ has a differential primary decomposition. We can choose the sets $\mathcal{A}_1,\dotsc,\mathcal{A}_s$ such that $|\mathcal{A}_i|$ is the arithmetic length of $M$ along the prime $P_i$.
\end{theorem}

\begin{proof}[Proof and discussion]
  The result is proved in \cite{CS} and further refined   in \cite{CC}.
  These sources also develop an algorithm. We shall explain this
   in Section \ref{sec5}, along with a discussion of the
     {\tt Macaulay2} command \texttt{solvePDE},
     which computes  differential primary decompositions.
  \end{proof}

The differential operators in $\mathcal{A}_1,\ldots,\mathcal{A}_s$ are known
as {\em Noetherian operators}  in the literature;
see \cite{CHKL, CPS, DAMIANO, oberst99}.
 Theorem \ref{thm:diff_prim_dec_existence} says that we can find a collection of
 $\,\operatorname{amult}(M) = m_1 + \cdots + m_s\,$ Noetherian operators
 in $D_n^k$ to 
characterize membership in the module $M$.

\begin{remark}
The construction of Noetherian operators is studied in $\cite{BJORK, 
CCHKL, CHKL, CPS,  HORMANDER, oberst99}$. Some of these sources
offer explicit methods, while others remain at an abstract level.
All previous methods share one serious shortcoming, namely they
yield operators separately~for each primary component $M_i$ of $M$.
They do not take into account how one primary component is embedded into another.
This leads to a number of
 operators that can be much larger than
  $\operatorname{amult}(M)$. We refer to \cite[Example 5.6]{CS} for an
  instance from algebraic statistics where the previous methods 
  require $ 1044$ Noetherian operators, while
  ${\rm amult}(M) = 207$~suffice.
\end{remark}  

While Theorem \ref{thm:diff_prim_dec_existence} makes no claim of minimality,
it is known that $\operatorname{amult}(M)$ is the minimal number of Noetherian
operators required for a differential primary decomposition of a certain desirable form.
To make this precise, we begin with a few necessary definitions.
For any given subset $\mathcal{S}$ of $\{x_1,\ldots,x_n\}$,
 the \emph{relative Weyl algebra} is defined as the subring of the Weyl algebra 
 $D_n$ 
 using only differential operators corresponding to variables not in $\mathcal{S}$:
\begin{equation}
\label{eq:usingonly}
  D_n(\mathcal{S}) \,:= \, R \langle \,\partial_{x_i} \colon x_i \not \in \mathcal{S} \,\rangle
  \,\, \subseteq \,\, D_n.
\end{equation}
Thus, if $\mathcal{S} = \emptyset$ then $D_n(\mathcal{S}) = D_n$,
and if $\mathcal{S} = \{x_1,\ldots,x_n\}$ then $D_n(\mathcal{S}) = R
= K[x_1,\ldots,x_n]$.

For any prime ideal $P_i$ in $R$ we fix a 
set $\mathcal{S}_i \subseteq \{x_1,\ldots,x_n\}$
that satisfies $K[\mathcal{S}_i] \cap P_i = \{0\}$
and is maximal with this property.
Thus, $\mathcal{S}_i$ is a maximal independent set 
of coordinates on the irreducible variety $V(P_i)$.
Equivalently, $\mathcal{S}_i$ is a basis of
the algebraic matroid defined by the prime $P_i$; cf.~\cite[Example 13.2]{INLA}.
The cardinality of $\mathcal{S}_i$ equals the dimension of  $V(P_i)$.

\begin{theorem} \label{thm:DPD}
The differential primary decomposition in Theorem~\ref{thm:diff_prim_dec_existence}
can be chosen so that $\mathcal{A}_i \subset D_n(\mathcal{S}_i)^k$.
The arithmetic length of $M$ along $P_i$ is a lower bound for the cardinality of $\mathcal{A}_i$
in any differential primary decomposition of~$M$ such that $\mathcal{A}_i \subset D_n(\mathcal{S}_i)^k$ for  $i = 1,\dotsc,s$.
\end{theorem}

\begin{proof}[Proof and discussion]
This was shown in \cite[Theorem 4.6]{CS}. The case of ideals 
$(k=1)$ appears in \cite[Theorem 3.6]{CS}. See also \cite{CC}.
The theory developed in \cite{CS} is 
more general in that $R$ can be any Noetherian $K$-algebra.
In this paper we restrict to  polynomial rings $R = K[x_1,\ldots,x_n]$ 
where $K$ is a subfield of $\CC$.
That case is treated in detail in \cite{CC}.
\end{proof}

We next argue that Theorems~\ref{thm:Palamodov_Ehrenpreis} and 
\ref{thm:diff_prim_dec_existence}
are really two sides of the same coin.
Every element $A$ in the Weyl algebra $D_n$ acts as a differential operator with polynomial
coefficients on functions in the unknowns ${\bf x} = (x_1,\ldots,x_n)$.
Such a differential operator has a unique representation where 
all derivatives are moved to the right of the polynomial coefficients:
\begin{equation}
\label{eq:Ann}
\qquad A({\bf x},\partial_{\bf x}) \,\,\, = \,\, \sum_{{\bf r},{\bf s} \in \NN^n} \!\! c_{{\bf r},{\bf s}} \,
x_1^{r_1} \cdots x_n^{r_n} \partial_{x_1}^{s_1} \cdots \partial_{x_n}^{s_n} , \qquad
{\rm where} \,\,c_{{\bf r},{\bf s}} \in K . 
\end{equation}
There is a natural $K$-linear isomorphism between the Weyl algebra $D_n$ and the
polynomial ring $K[{\bf x},{\bf z}]$ which takes the operator $A$ in (\ref{eq:Ann})
to the following polynomial $B$ in $2n $ variables:
\begin{equation}
\label{eq:Bnn}
 B({\bf x},{\bf z}) \,\,\, = \,\, \sum_{{\bf r},{\bf s} \in \NN^n} \!\! c_{{\bf r},{\bf s}} \,
x_1^{r_1} \cdots x_n^{r_n} z_1^{s_1} \cdots z_n^{s_n} . \qquad \qquad \quad
\end{equation}

In Sections \ref{sec1}, \ref{sec2} and \ref{sec3}, polynomials
 in $R =  K[x_1,\ldots,x_n] = K[\partial_1,\ldots,\partial_n] $
act as differential operators  on functions in the unknowns ${\bf z} = (z_1,\ldots,z_n)$.
For such operators, polynomials in ${\bf x}$ are constants.
By contrast, in the current section, we introduced the Weyl algebra $D_n$.
Its elements
act on functions in ${\bf x} = (x_1,\ldots,x_n)$, with polynomials in ${\bf z}$ 
being constants. These two different actions of
differential operators, by $D_n$ and $R$ on scalar-valued functions,
extend to actions by $D_n^k$ and $R^k$ on vector-valued functions. 
We highlight the following key~point:
\begin{equation}
\label{eq:absolutelyessential}
\hbox{\em Our distinction between the ${\bf z}$-variables and ${\bf x}$-variables is
absolutely essential.}
\end{equation}

The following theorem is the punchline of this section.
It allows us to identify Noetherian operators (\ref{eq:Ann}) with
Noetherian multipliers (\ref{eq:Bnn}). 
This was assumed tacitly in \cite[Section 3]{CPS}.

\begin{theorem} \label{thm:DPD_to_integrals}
Consider any differential primary decomposition of the module $M$
as in Theorem~\ref{thm:DPD}. Then this translates into 
an Ehrenpreis--Palamodov representation of 
the solution space ${\rm Sol}(M)$. Namely, if we replace
each operator $A({\bf x},\partial_{\bf x})$ in $\mathcal{A}_i$ by the 
corresponding polynomial $B({\bf x},{\bf z})$, then these
${\rm amult}(M)$
polynomials satisfy the conclusion of 
Theorem~\ref{thm:Palamodov_Ehrenpreis}.
\end{theorem}

\begin{example}[$k=l=n=1$] \label{ex:warning}
We illustrate Theorem \ref{thm:DPD_to_integrals}
and the warning  (\ref{eq:absolutelyessential}) 
for an ODE (\ref{eq:ODE}) with $m=3$.
Set $\, p(x)\, =\,  x^3 + 3 x^2 - 9x + 5 \, = \,(x-1)^2 (x+5)\,$
in (\ref{eq:charpoly}).
The ideal $I = \langle\, p\, \rangle$  has
 $s=2$ associated primes in $R = \QQ[x] $, namely
$P_1 = \langle x-1 \rangle $ and $P_2 = \langle x+5 \rangle$,
with $m_1 = 2$ and $m_2=1$, so ${\rm amult}(I) = 3$.
A differential primary decomposition of $I$ is given by
$\mathcal{A}_1 = \{1,\partial_x \}$ and $\mathcal{A}_2 = \{1\}$.
The three Noetherian operators translate into the Noetherian multipliers
$B_{11} = 1, B_{12} = z, B_{21} = 1$. The integrals in (\ref{eq:anysolution})
now furnish the general solution
$ \phi(z) = \alpha \,{\rm exp}(z) + \beta z \,{\rm exp}(z) + \gamma\, {\rm exp}(-5z) $
to the differential equation $\phi''' + 3 \phi'' - 9 \phi' + 5 \phi = 0$.
\end{example}

The derivation of  Theorem \ref{thm:DPD_to_integrals} 
rests on the following lemma on duality between ${\bf x}$ and ${\bf z}$.

\begin{lemma}\label{lem:poly_exp_duality}
Let $p$ and $q$ be polynomials in $n$ unknowns with coefficients in $K$. We have
  \begin{equation}
  \label{eq:pqcommute}
    q(\partial_{\bf z}) \bullet \bigl(p({\bf z}) \exp({\bf x}^t {\bf z}) \bigr)
    \,\, =\,\, p(\partial_{\bf x}) \bullet \bigl( q({\bf x}) \exp({\bf x}^t {\bf z}) \bigr).    
\end{equation}
\end{lemma}
\begin{proof}
The parenthesized expression on the left
equals $p(\partial_{\bf x}) \bullet {\rm exp}({\bf x}^t {\bf z})$, while that on the right
equals $q(\partial_{\bf z}) \bullet {\rm exp}({\bf x}^t {\bf z})$.
Therefore the expression in (\ref{eq:pqcommute}) is the
result of applying the operator
$p(\partial_{\bf x}) q(\partial_{\bf z}) = q(\partial_{\bf z}) p(\partial_{\bf x}) $ 
to $ \exp({\bf x}^t {\bf z}) $, when viewed as a function in $2n$ unknowns.
\end{proof}

We now generalize this lemma to $k \geq 2$, we replace
$p$ by a polynomial vector that depends on both ${\bf x}$ and ${\bf z}$,
and we rename that vector using the identification between
(\ref{eq:Ann}) and~(\ref{eq:Bnn}).

\begin{proposition}\label{prop:noeth_ops_module_duality}
  Let $B({\bf x},{\bf z})$   be a $k$-tuple of polynomials in $2n$ variables 
  and $A({\bf x},\partial_{\bf x}) \in D_n^k$ the corresponding
  $k$-tuple of differential operators in the Weyl algebra.
    Then we have
  \begin{align}
  q(\partial_\mathbf{z}) \bullet (B(\mathbf{x}, \mathbf{z}) \exp(\mathbf{x}^t \mathbf{z}))
\,\, = \,\,
    A(\mathbf{x}, \partial_\mathbf{x}) \bullet (q(\mathbf{x}) \exp(\mathbf{x}^t \mathbf{z})) .  \end{align}
\end{proposition}

\begin{proof}
  If $k=1$, we write $A(\mathbf{x}, \partial_\mathbf{x}) 
  = \sum_{\alpha} c_\alpha(\mathbf{x}) \partial_\mathbf{x}^\alpha$
  as in (\ref{eq:Ann}) and
 $B({\bf x},{\bf z}) = \sum_{\alpha} c_\alpha(\mathbf{x}) \mathbf{z}^\alpha$
 as in (\ref{eq:Bnn}).   Only finitely many
of  the polynomials $c_\alpha(\mathbf{x})$ are nonzero.
  Applying Lemma \ref{lem:poly_exp_duality} gives
  $$ \begin{matrix}
    A(\mathbf{x}, \partial_\mathbf{x}) \bullet (q(\mathbf{x}) \exp(\mathbf{x}^t \mathbf{z}))
    \,\, = \,\,\sum_\alpha c_\alpha(\mathbf{x}) q(\partial_\mathbf{z}) \bullet ({\bf z}^\alpha \exp(\mathbf{x}^t \mathbf{z})) \,\,=
    \,\, q(\partial_\mathbf{z}) \bullet (B(\mathbf{x}, \mathbf{z}) \exp(\mathbf{x}^t \mathbf{z})).
    \end{matrix}
$$
  The extension from $k=1$ to $k \geq 2$ follows because
  the differential operation $\bullet$ is $K$-linear.
  \end{proof}

We now take a step towards proving Theorem \ref{thm:DPD_to_integrals}
in the case $s=1$. Let $M$ 
be a primary submodule of $R^k$ with $\operatorname{Ass}(M) = \{P\}$.
Its support $V(M) = V(P)$ is an  irreducible affine variety in $\CC^n$.
Consider  the sets of all Noetherian operators and all Noetherian multipliers:
\begin{equation}
\label{eq:frakAB}
  \begin{matrix}
\mathfrak{A} & := & \,\,\, \bigl\{\, A \in D_n^k\,:\,
 A \bullet m \in P \,\,\hbox{for all} \,\, m \in M \bigr\} \qquad \qquad  
 \qquad \qquad \qquad {\rm and} \\
  \mathfrak{B} & := &\!\!\!\! \!\!\bigl\{\,B \in K[{\bf x},{\bf z}]\,:\,
 B({\bf x},{\bf z}) \,{\rm exp}({\bf x}^t {\bf z}) \in {\rm Sol}(M)\,\,\,
\hbox{for all $\,{\bf x} \in V(P) $} \bigr\}. 
\end{matrix}
\end{equation}

\begin{proposition} \label{prop:NONM}
The bijection between $D_n^k$ and $K[{\bf x},{\bf z}]^k$, given by
identifying the operator $A$ in (\ref{eq:Ann}) with the polynomial
$B$ in (\ref{eq:Bnn}), restricts to a bijection
between the sets $\mathfrak{A}$ and $\mathfrak{B}$.
\end{proposition}

\begin{proof}
Let $m_1,\ldots,m_l \in K[{\bf x}]^k$ be  generators of $M$.
Suppose $A \in \mathfrak{A}$. Then
  \begin{align*}
    \sum_{i = 1}^k A_i(\mathbf{x}, \partial_\mathbf{x}) \bullet \sum_{j = 1}^l m_{ij}(\mathbf{x}) f_j(\mathbf{x})
  \end{align*}
  vanishes for all $\mathbf{x} \in V(P)$ and all polynomials $f_1,\dotsc,f_l \in \CC[\mathbf{x}]$.
Since the space of complex-valued polynomials is dense in the space of all entire functions
on $\CC^n$,
the preceding implies
$$
    \sum_{i = 1}^k A_i(\mathbf{x}, \partial_\mathbf{x}) \bullet 
    m_{ij}(\mathbf{x}) \exp(\mathbf{x}^t\mathbf{z}) \,\,= \,\, 0 \quad
    \hbox{for all $\mathbf{z} \in \CC^n$, $\mathbf{x} \in V(P)$ and $j = 1,\dotsc,l$.}
$$
  Using Proposition \ref{prop:noeth_ops_module_duality}, this yields
  $$
   \sum_{i = 1}^k  m_{ij}(\partial_\mathbf{z}) \bullet B_i(\mathbf{x}, \mathbf{z}) \exp(\mathbf{x}^t\mathbf{z}) 
\,\,= \,\, 0 \quad
    \hbox{for all $\mathbf{z} \in \CC^n$, $\mathbf{x} \in V(P)$ and $j = 1,\dotsc,l$.}
 $$
 We conclude that  the polynomial
  vector $B(\mathbf{x},\mathbf{z}) $ corresponding to $A({\bf x},\partial_{\bf x})$ 
  lies in $\mathfrak{B}$.

To prove the converse, we note that the implications above are reversible.
Thus, if $B({\bf x},{\bf z}) $ is in $\mathfrak{B}$ then
$A({\bf x},\partial_{\bf x})$ is in $\mathfrak{A}$.
This uses the fact that  linear combinations of~the  exponential functions 
$\mathbf{x} \to \exp(\mathbf{x}^t \mathbf{z})$,
for ${\bf z} \in \CC^n$, are also dense in the space of entire functions.
\end{proof}

\begin{proof}[Proof of Theorem \ref{thm:DPD_to_integrals}]
Let $\mathcal{A}$ be any finite subset of $\mathfrak{A}$ 
which gives a differential primary decomposition of the $P$-primary module
$M$. This exists and can be chosen to have cardinality equal to the length of $M$
along $P$.  Let $\mathcal{B}$ be the set of Noetherian multipliers (\ref{eq:Bnn}) 
corresponding to the set $\mathcal{A}$ of Noetherian operators (\ref{eq:Ann}).
Proposition \ref{prop:NONM} shows that the exponential function
$\,{\bf z} \to B({\bf x},{\bf z}) \,{\rm exp}({\bf x}^t {\bf z})$ is in 
 ${\rm Sol}(M)$ whenever ${\bf x} \in V(P)$ and  $B \in \mathcal{B}$.
 Hence all $\CC$-linear combinations of such functions are in ${\rm Sol}(M)$.
 More generally, by differentiating under the integral sign, we find that
 all functions of the following form are solutions of $M$:
$$ \psi({\bf z}) \,\, = \,\,
\sum_{B \in \mathcal{B}} \int_{V(P)} B({\bf x},{\bf z})\, {\rm exp}({\bf x}^t {\bf z}) \, d\mu_B({\bf x}).
$$

We need to argue  that all solutions in $\mathcal{F} = C^\infty(\Omega)$
admit such an integral representation.
Suppose first that all associated primes of $M$ are minimal.
Then each $\mathcal{A}_i$ spans a bimodule in the sense of
\cite[Theorem 3.2 (d)]{CC}. Hence,
 for each associated prime $P_i$, the~module 
\begin{align*}
  M_i \,\,=\,\, \{ m \in R^k \colon \delta \bullet m \in P_i \text{ for all } \delta \in \mathcal{A}_i \}
\end{align*}
is $P_i$-primary,  and
  $M = M_1 \cap \dotsb \cap M_s$ is a minimal primary decomposition.
The operators in $\mathcal{A}_i$ are in the relative Weyl algebra $D_n(\mathcal{S}_i)$ and fully characterize the $P_i$-primary component of $M$.
We may thus follow the classical analytical constructions in the books \cite{BJORK, HORMANDER, PALAMODOV} to patch together the integral representation
of ${\rm Sol}(M_i)$ for $i=1,\ldots,s$,
under the correspondence of Noetherian operators and Noetherian multipliers.
Therefore, all solutions have the form~\eqref{eq:anysolution}.

Things are more delicate when $M$ has embedded primes.
Namely,  if $P_i$ is embedded then the operators in $\mathcal{A}_i$ only characterize the contribution of the $P_i$-primary component relative to all other components contained in $P_i$. 
We see this in Section \ref{sec5}.
One argues by enlarging $\mathcal{A}_i$ to vector space generators of the relevant bimodule.
Then the previous patching argument applies. And, afterwards one shows
that the added summand in the integral representation are redundant because they
are covered by associated varieties $V(P_j)$ containing $V(P_i)$.
\end{proof}

\section{Software and Algorithm}
\label{sec5}

In this section we  present an algorithm for solving linear PDE with constant coefficients.
It is based on the methods for ideals given in
\cite{CCHKL, CHKL, CPS}. The case of modules appears in~\cite{CC}.
We note that the computation of Noetherian operators
has a long history, going back to work in the 1990's by Ulrich Oberst 
\cite{oberst90,oberst95, oberst96, oberst99}, who developed a construction of Noetherian
operators for primary modules.
This was further developed by
Damiano, Sabadini and Struppa \cite{DAMIANO}
who presented the first Gr\"obner-based algorithm.
It works for primary ideals under the restrictive assumption that the characteristic variety has a rational point after passing to a (algebraically non-closed) field of fractions.  Their article also points to an implementation in {\tt CoCoA}, 
but we were unable to access that~code.
Since these early approaches rely on the ideals or modules being primary, using them in practice requires first computing a primary decomposition.
If there are embedded primes, the number of Noetherian operators output by these methods will not be minimal either.

We here present a new algorithm that is universally applicable, to all
ideals and modules over a polynomial ring.
There are no restrictions on the input and the output is minimal.
The input is a submodule
$M$ of $R^k$, where $R = K[x_1,\ldots,x_n]$.
The output is a differential primary decomposition of size ${\rm amult}(M)$
as in Theorem \ref{thm:DPD}. A first step~is to find
${\rm Ass}(M) = \{P_1,\ldots,P_s\}$. For each
associated prime $P_i$,  the elements $A({\bf x},\partial_{\bf x})$
in the finite set $\mathcal{A}_i \subset D_n(\mathcal{S}_i)$ are 
rewritten as polynomials $B({\bf x},{\bf z})$, using 
the identification of (\ref{eq:Ann}) with
(\ref{eq:Bnn}). Only the ${\rm codim}(P_i)$ many variables $z_i $
with $x_i \not\in \mathcal{S}_i$
appear in these Noetherian multipliers~$B$.

We now describe our implementation for (\ref{eq:task2})
in {\tt Macaulay2} \cite{MAC2}.
The command is called {\tt solvePDE}, as in \cite[Section 5]{CS}.
It is distributed
with \texttt{Macaulay2} starting from version 1.18 in
the package \texttt{NoetherianOperators} \cite{CCHKL}.
The user begins by fixing a polynomial ring $R = K[x_1,\ldots,x_n]$.
Here $K$ is usually the rational numbers~${\tt QQ}$. Fairly
arbitrary variable names $x_i$ are allowed. The argument of {\tt solvePDE}
is an ideal in $R$ or a submodule of $R^k$.
The output is a list of pairs $\bigl\{P_i,\{B_{i1},\ldots,B_{i,m_i}\} \bigr\}$
for $i=1,\ldots,s$, where $P_i$ is a prime ideal given by
generators in $R$, and each $B_{ij}$ is a vector over
a newly created polynomial  ring $K[x_1,\ldots,x_n,z_1,\ldots,z_n]$.
The new variables $z_i$ are named internally by
{\tt Macaulay2}. The system writes ${\tt d} x_i$ for $z_i$. 
To be precise, each new variable is created from an 
old variable by prepending the character ${\tt d}$.
This notation can be confusing at first, but one gets used to it.
The logic comes from the
differential primary decompositions described in \cite[Section~5]{CS}.

Each $B_{ij}$ in the output of {\tt solvePDE} encodes
an exponential solution  $\,B_{ij}({\bf x},{\bf z}) \,{\rm exp}({\bf x}^t {\bf z})\,$
to $M$. Here ${\bf x}$ are the old variables chosen 
by the user, and ${\bf x}$ denotes points in the irreducible
variety $V(P_i) \subseteq \CC^n$. The solution is a function in the new unknowns
${\bf z} = ({\tt d}x_1,\ldots,{\tt d}x_n)$.
For instance, if $n=3$ and the input is in 
the ring ${\tt QQ[u,v,w]}$ then the output
lives in the ring ${\tt QQ[u,v,w,du,dv,dw]}$.
Each solution to the PDE is a function
$\psi({\tt du},{\tt dv},{\tt dw})$ and these
functions are parametrized by a variety $V(P_i)$ in a $3$-space 
whose coordinates are $({\tt u},{\tt v},{\tt w})$.

We now demonstrate how this works for two examples
featured in the introduction.

\begin{example}  \label{ex:51}
Consider the third order ODE (\ref{eq:ODE}) in Example  \ref{ex:warning}.
We solve this as follows:
\begin{verbatim}
 R = QQ[x]; I = ideal( x^3 + 3*x^2 - 9*x + 5 ); solvePDE(I)
  {{ideal(x - 1), {| 1 |, | dx |}}, {ideal(x + 5), {| 1 |}}}
 \end{verbatim}
 \vspace{-0.2in}
 The first line is the input.
 The second line is the output created by {\tt solvePDE}.
  This list of $s=2$ pairs encodes the general solution $\phi(z)$.
  Remember: $z$ is the newly created symbol {\tt dx}.
\end{example}


\begin{example} \label{ex:53}
We solve the PDE (\ref{eq:pdesys1}) by typing the $ 2 \times 3$ matrix whose
columns are (\ref{eq:threevectors}):
\begin{verbatim}
 R = QQ[x1,x2,x3,x4];
 M = image matrix {{x1^2,x2*x3,x1^2*x3},{x1*x2,x3^2,x1*x2*x4}}; solvePDE(M)
 \end{verbatim}
 \vspace{-0.3in}
 The reader is encouraged to run this code, and to check that the output
       is  the solution  (\ref{eq:pdesys2}).
\end{example}

The method in {\tt solvePDE} is described in
Algorithm~\ref{alg:solvePDE} below.
A key ingredient is a translation map. We now
explain this in the simplest case, when the module is supported in one point.
Suppose
$V(M) =  \{\mathbf{u}\}$ for some $\mathbf{u} \in K^n$. We set
 $ \mathfrak{m}_\mathbf{u} = \langle x_1 - u_1, \dotsc, x_n - u_n \rangle$ and
 \begin{equation} \label{eq:gamma_u}
	\gamma_\mathbf{u} : R \to R \,,\,\,\,
	x_i \mapsto x_i + u_i \quad\text{ for }\,\, i=1,\dotsc,n.
	\end{equation}
	The following two results are straightforward.
	We will later use them when $M$ is any primary module,
	 ${\bf u}$ is the generic point of $V(M)$,
	and $\KK = K({\bf u})$ is  the associated field extension of~$K$.

\begin{proposition}\label{prop:translation}
A constant coefficient operator
$A(\partial_\mathbf{x})$ is a Noetherian operator for the $\mathfrak{m}_\mathbf{u}$-primary module $M$ if and only if $A(\partial_\mathbf{x})$ is a Noetherian operator for the $\mathfrak{m}_0$-primary  
	 module $\hat M \coloneqq \gamma_\mathbf{u}(M)$.
	 Dually, $B(\mathbf{z}) \exp(\mathbf{u}^t\mathbf{z}) \,$
is in $\, \operatorname{Sol}(M)$ if and only if
$B(\mathbf{z})$
is in $\, \operatorname{Sol}(\hat M)$.
\end{proposition}

We note that
all Noetherian operators over a $K$-rational point can be taken to have constant coefficients.
This follows from Theorem \ref{thm:finitedim}. This observation
 reduces the computation of solutions for a primary module
 to finding the polynomial solutions of the translated module. Next, we 
  bound the degrees of these polynomials.

\begin{proposition}\label{prop:poly_sol_deg_bound}
Let $\hat M \subseteq R^k$ be an $\,\mathfrak{m}_0$-primary module.
There exists an integer $r$ such that $\,\mathfrak{m}_0^{r+1} R^k \subseteq \hat M$. The 
space $\operatorname{Sol}(\hat M)$ consists of $k$-tuples of polynomials of degree $\leq r$.
\end{proposition}

Propositions \ref{prop:translation} and \ref{prop:poly_sol_deg_bound} 
furnish a method for  computing solutions of an
$\mathfrak{m}_\mathbf{u}$-primary module $M$.
We start by translating $M$ so that it becomes the $\mathfrak{m}_0$-primary module $\hat M$.
The integer $r$ provides an ansatz 
$ \sum_{j=1}^k\sum_{|\alpha| \leq r} v_{\alpha,j} \,\mathbf{z}^\alpha e_j $
for the polynomial solutions. The coefficients $v_{\alpha,j}$
are computed by linear algebra over the ground field $K$.
Here are the  steps:
\begin{enumerate}
	\item Let $r$ be the smallest integer such that 
	$\mathfrak{m}_0^{r+1} R^k \subseteq \hat M$.
\item Let $\operatorname{Diff}( \hat M)$ be the matrix whose entries are the
polynomials $\hat m_i \bullet ({\bf z}^\alpha e_j) \in R$.
The row labels are the generators $\hat m_1, \dotsc, \hat m_l $ of $\hat M$, and the column labels
are the ${\bf z}^\alpha e_j$.
	\item 
    Let $\ker_K(\operatorname{Diff}(\hat M))$ denote the
    $K$-linear subspace of the $R$-module $\ker_R(\operatorname{Diff}(\hat M))$ 
    consisting of vectors $(v_{\alpha,j}) $   with all entries in $K$.
Every such vector
gives a solution
\begin{align}\label{eq:step_4_identification}
	\sum_{j=1}^k \sum_{|\alpha| \leq r} v_{\alpha, j}\, {\bf z}^\alpha 
	\exp(\mathbf{u}^t \mathbf{z}) \,e_j \,\in \, \operatorname{Sol}(M).
\end{align}
\end{enumerate}

\begin{example}$[n=k=r=2]$ \label{ex:quotex1} \ \ The
following module is $\mathfrak{m}_0$-primary of multiplicity three:
\begin{equation}
\label{eq:quotexIN}
M \,\,\, = \,\,\, {\rm image}_R \begin{bmatrix} 
\,\,\partial_1^3 \,&\, \,\partial_2 - c_1 \partial_1^2 - c_2 \partial_1 \,& \,
\, c_3 \partial_1^2 + c_4 \partial_1 + c_5\, \,\\
\,\, 0 & 0 & 1 \\
\end{bmatrix}.
\end{equation}
Here $c_1,c_2,c_3,c_4,c_5$ are arbitrary constants in $K$.
The matrix ${\rm Diff}(M)$ has three rows, one
for each generator of $M$, and it has $12$ columns,
indexed by $e_1,z_1e_1,\ldots , z_2^2 e_1,
e_2,z_1e_2,\ldots , z_2^2 e_2$.
The space ${\rm ker}_K({\rm Diff}(M))$ is $3$-dimensional.
A basis furnishes the three polynomial solutions
\begin{equation}
\label{eq:quotexOUT}
\begin{bmatrix}
\,-1\, \\ \, c_5,
\end{bmatrix} \, ,\,\,\,
\begin{bmatrix}
  -(z_1+c_2 z_2) \\
  \,   c_5 z_1+c_2 c_5 z_2+c_4 \,\, \end{bmatrix}
\, , \,\,\,
\begin{bmatrix}
  -((z_1+c_2 z_2)^2 + 2 c_1 z_2)  \\
\, c_5 (z_1{+}c_2 z_2)^2 + 2 c_4z_1 + 2 (c_1 c_5{+}c_2 c_4)z_2  + 2c_3\,
 \end{bmatrix}.
\end{equation}
\end{example}

We now turn to Algorithm \ref{alg:solvePDE}.
The input and output are as described in (\ref{eq:task2}).
The method  was introduced in \cite[Algorithm 4.6]{CC} 
for computing differential primary decompositions.
We use it
for solving PDE.
It is implemented in {\tt Macaulay2} under the command {\tt solvePDE}.
In our discussion,  the line numbers refer to the corresponding lines of pseudocode in Algorithm~\ref{alg:solvePDE}.

\begin{algorithm}
\hspace*{\algorithmicindent} \textbf{Input:} An arbitrary submodule $M$ of $R^k$ \\
\hspace*{\algorithmicindent} \textbf{Output:} List of associated primes with corresponding Noetherian multipliers.
\caption{SolvePDE}
\begin{algorithmic}[1]
\For{each associated prime ideal $P$ of $M$}\label{alg:line1}
  \State $U\gets MR_P^k\cap R^k$\label{alg:line2}
  \State $V\gets(U:P^\infty)$\label{alg:line3}
  \State $r \gets $ the smallest number such that $V\cap P^{r+1}R^k$ is a subset of $U$\label{alg:line4}
  \State $\mathcal{S} \gets $ a maximal set of independent variables modulo $P$\label{alg:line5}
  \State $\mathbb{K} \gets \operatorname{Frac}(R/P)$\label{alg:line6}
  \State $T \gets \mathbb{K}[{y_i} \colon x_i \not \in \mathcal{S}]$\label{alg:line7}
  \State $\gamma \gets$ the map defined in \eqref{gamma}\label{alg:line8}
  \State $\mathfrak{m} \gets $ the irrelevant ideal in $T$\label{alg:line9}
  \State $\hat{U}\gets\gamma(U)+\mathfrak{m}^{r+1}T^k$\label{alg:line10}
  \State $\hat{V}\gets\gamma(V)+\mathfrak{m}^{r+1}T^k$\label{alg:line11}
  \State $N \gets$ a $\mathbb{K}$-vector space basis of the space of $k$-tuples of polynomials of degree $\leq r$\label{alg:line12}
  \State $\operatorname{Diff}(\hat{U}) \gets$ the matrix given by the $\bullet$-product of generators of $\hat{U}$ with elements of $N$\label{alg:line13}
  \State  $\operatorname{Diff}(\hat{V}) \gets$ the matrix given by the $\bullet$-product of generators of $\hat{V}$ with elements of $N$\label{alg:line14}
  \State $\mathcal{K} \gets \ker_{\mathbb{K}}(\operatorname{Diff}(\hat{U}))/\ker_{\mathbb{K}}(\operatorname{Diff}(\hat{V}))$\label{alg:line15}
  \State $\mathcal{A} \gets $ a $\mathbb{K}$-vector space basis of  $\mathcal{K}$\label{alg:line16}
  \State $\mathcal{B} \gets $ lifts of the vectors in $\mathcal{A} \subset T^k$ to vectors in $R[\mathtt{d}x_1,\dotsc,\mathtt{d}x_n]^k$\label{alg:line17}
  \State \Return the pair $(P, \mathcal{B})$\label{alg:line18}
\EndFor
\end{algorithmic}\label{alg:solvePDE}
\end{algorithm}

\begin{description}
    \item[Line \ref{alg:line1}]
    We begin by finding all associated primes of $M$. 
    These define the irreducible   varieties $V_i$ in (\ref{eq:task2}).
      By \cite[Theorem 1.1]{EHV}, the associated primes of codimension $i$ coincide with the minimal primes of  $\operatorname{Ann} \operatorname{Ext}^i_R(M,R)$.
      This reduces the problem of finding associated primes of a module to the more familiar problem of finding minimal primes of a polynomial ideal.
      This method is implemented and distributed with {\tt Macaulay2} starting from version 1.17 via the command \texttt{associatedPrimes R\^{}k/M}.
      See \cite[Section~2]{CC}.    
\end{description}

The remaining steps are repeated for each $P \in \operatorname{Ass}(M)$.
For a fixed associated prime $P$, our goal is to 
identify the contribution to $\operatorname{Sol}(M)$ of the $P$-primary component of $M$.

\begin{description}

    \item[Lines \ref{alg:line2}--\ref{alg:line3}]
      To achieve this goal, we study solutions for two different $R$-submodules of $R^k$.
      The first one, denoted $U$, is the intersection of all $P_i$-primary components of $M$, where $P_i$ are the associated primes contained in $P$. Thus $U = MR^k_P \cap R^k$, which
  is the extension-contraction module of $M$ under localization at $P$.   
 It is computed as $U = (M : f^\infty)$, where $f \in R$ is 
 contained in every associated prime $P_j$ \emph{not} contained in~$P$.

      The second module, denoted $V$, is the intersection of all $P_i$-primary components of $M$, where $P_i$ is \emph{strictly} contained in $P$.
      Hence $V = (U : P^\infty)$ is the saturation of $U$ at $P$.
      We have $U = V \cap Q$, where $Q$ is a $P$-primary 
      component of $M$.
            Thus the difference between the solution spaces $\operatorname{Sol}(U)$ and $\operatorname{Sol}(V)$ is caused by the primary module $Q$.

When $P$ is a minimal prime, $U$ is the unique $P$-primary component of $M$, and $V = R^k$.
      
    \item[Line \ref{alg:line4}] 
      The integer $r$ bounds the  degree of Noetherian multipliers associated to $U$ 
      but not $V$.
    Namely,  if  the function $\phi(\mathbf{z}) = B(\mathbf{x}, \mathbf{z}) \exp(\mathbf{x}^t \mathbf{z})$ lies in $\operatorname{Sol}(U) \backslash \operatorname{Sol}(V)$
      for all $\mathbf{x} \in V(P)$,
       then the $\mathbf{z}$-degree of the polynomial $B(\mathbf{x}, \mathbf{z})$ 
       is at most $r$.
    This will  lead to an ansatz for the Noetherian multipliers 
  responsible for the difference between $\operatorname{Sol}(U)$ and $\operatorname{Sol}(V)$.

    \item[Lines \ref{alg:line5}--\ref{alg:line8}]  The modules $U$ and $V$ are
   reduced to simpler modules $\hat U$ and $\hat V$ with similar properties.
    Namely, $\hat U$ and $\hat V$
    are primary and their characteristic varieties are the origin.
    This reduction involves two new ingredients: a new polynomial ring $T$ in fewer variables 
   over a field $\mathbb{K}$ that is a finite extension of $K$, and a 
   ring map $\gamma \colon R \to T$.
    
    Fix a maximal set $\mathcal{S} = \{x_{i_1},\ldots,x_{i_{n-c}}\}$ with $P\cap K[x_{i_1},\ldots,x_{i_{n-c}}]=\{0\}$. We define $T\coloneqq \mathbb{K}[y_i: x_i\notin \mathcal{S}]$, where $\mathbb{K}=\operatorname{Frac}(R/P)$.
    This is a polynomial ring in $n - |\mathcal{S}| = c$ new variables $y_i$, 
    corresponding to the $x_i$ not in the set $\mathcal{S}$ of independent variables.
Writing $u_i$ for the image of $x_i$ in $\mathbb{K}=\operatorname{Frac}(R/P)$,
the ring map $\gamma$ is defined  as follows:
    \begin{equation}\label{gamma}\gamma \colon R \to T, \quad x_i\mapsto \begin{cases} y_i+u_i, & \text{ if } x_i\notin S,\\
     \quad u_i, & \text{ if } x_i\in S.
    \end{cases} \end{equation}
    By abuse of notation, we denote by $\gamma$ the extension of 
    (\ref{gamma}) to a map $R^k \to T^k$. 

  \item[Lines \ref{alg:line9}--\ref{alg:line11}] Let $\mathfrak{m} := \langle y_i \colon x_i \not \in \mathcal{S} \rangle $
   be the irrelevant ideal of $T$.
    We define the $T$-submodules 
$$      \hat U := \gamma(U) + \mathfrak{m}^{r+1} T^k  \quad {\rm and} \quad \hat V := \gamma(U) + \mathfrak{m}^{r+1} T^k
 \quad {\rm of} \,\,\,T^k.   $$
    These modules are $\mathfrak{m}$-primary: their solutions are finite-dimensional $\mathbb{K}$-vector~spaces consisting of polynomials of degree $\leq r$.
    The polynomials in $\operatorname{Sol}(\hat U) \backslash \operatorname{Sol}(\hat V)$ capture the difference between $\hat U$ and $\hat V$, 
    and also the difference between $U$ and $V$ after lifting.

  \item[Lines \ref{alg:line12}--\ref{alg:line14}] 
    We construct matrices $\operatorname{Diff}(\hat U)$ and $\operatorname{Diff}(\hat V)$ 
    with entries in $\mathbb{K} [z_i \colon x_i \not\in \mathcal{S}] $.  
    As in (\ref{eq:step_4_identification}), their kernels over $\KK$  correspond to 
        polynomial solutions of $\hat U$ and $\hat V$.
The set $N =  \{{\bf z}^\alpha e_j \colon |\alpha| \leq r, j=1,\dotsc, k\}$
is a   $\mathbb{K}$-basis for
elements of degree $\leq r$ in $\mathbb{K}[z_i \colon x_i \not\in \mathcal{S}]^k$.
    The $y_i$-variables act on the $z_i$ variables as partial derivatives, i.e. $y_i = \frac{\partial}{\partial z_i}$.
    We define the matrix $\operatorname{Diff}(\hat U)$ as follows.
    Let $\hat{U}_1,\ldots, \hat{U}_\ell$ be generators of $\hat{U}$. The rows of $\operatorname{Diff}(\hat{U})$ are indexed by these generators,  the columns are indexed by $N$, and the entries are 
    the polynomials $\hat{U}_i\bullet {\bf z}^\alpha e_j$.
    In the same way we construct $\operatorname{Diff}(\hat{V})$.

  \item[Lines \ref{alg:line15}--\ref{alg:line16}]
    Let $\ker_\mathbb{K}(\operatorname{Diff}(\hat U))$ be the space of vectors in the kernel of $\operatorname{Diff}(\hat U)$ whose entries are in $\mathbb{K}$.
  The $\mathbb{K}$-vector space
         $\ker_\mathbb{K}(\operatorname{Diff}(\hat U))$ 
    parametrizes the polynomial solutions
        \begin{align*}
      \sum_{j=1}^k \sum_{|\alpha| \leq r} v_{\alpha,j} \,{\bf z}^\alpha e_j 
      \,\, \in \,\, \operatorname{Sol}(\hat U).
    \end{align*}
    The same holds for $\hat V$.
    The quotient space $\mathcal{K}:= \ker_\mathbb{K}(\operatorname{Diff}(\hat U)) / \ker_\mathbb{K}(\operatorname{Diff}(\hat V))$ characterizes excess solutions in $\operatorname{Sol}(\hat U)$ relative to $\operatorname{Sol}(\hat V)$.
    Write $\mathcal{A}$ for a $\mathbb{K}$-basis of $\mathcal{K}$.

  \item[Lines \ref{alg:line17}--\ref{alg:line18}]
    We interpret  $\mathcal{A}$ as a set of Noetherian multipliers
    for $M$ by performing a series of lifts and transformations.
    For each element $\bar{\mathbf{v}} \in \mathcal{A}$, we choose a representative $\mathbf{v} \in \ker_\mathbb{K}(\operatorname{Diff}(\hat U))$.
    The entries of $\mathbf{v}$ are in $\mathbb{K} = \operatorname{Frac}(R/P)$, and may contain denominators.
    Multiplying $\mathbf{v}$ by a common multiple of the denominators yields a vector with entries in $R/P$,  indexed by $N$. We lift this to a 
     vector $\mathbf{u} = (u_{\alpha,j})$ with entries in $R$.
    The Noetherian multiplier corresponding to ${\bf u}$ is the following
    vector in \mbox{$R[\mathtt{d}x_i \colon x_i \not \in \mathcal{S}]^k$}:
    \begin{align*}
      B(\mathbf{x}, \mathbf{\mathtt{d}x}) \,\,=\,\,\, \sum_{j=1}^k \sum_{|\alpha| \leq r}
      \, u_{\alpha,j}(\mathbf{x})\, \mathbf{\mathtt{d}x}^\alpha e_j .
    \end{align*}
        Applying the map $\bar{\mathbf{v}} \mapsto \mathbf{u}$ to
  each $\bar{\mathbf{v}} \in \mathcal{A}$ yields a set $\mathcal{B}$ of Noetherian multipliers. These multipliers describe the contribution of the $P$-primary component of $M$ to ${\rm Sol}(M)$.
\end{description}

The output of Algorithm 
\ref{alg:solvePDE} is a list of pairs $(P, \mathcal{B})$, where $P$ ranges over
${\rm Ass}(M)$ and $\mathcal{B} = \{B_1,\ldots,B_m\}$ is a
subset of $R[\mathtt{d}x_1, \dotsc, \mathtt{d}x_n]^k$.
The cardinality $m$ is the multiplicity of $M$ along $P$.
The output describes the solutions to the PDE given by $M$. Consider the  functions 
$$\phi_P(\mathtt{d}x_1,\ldots,\mathtt{d}x_n)\,\,= \,\,
\sum_{i=1}^m \int_{V(P)}B_i(\mathbf{x}, \mathbf{\mathtt{d}x})\exp(x_1\,\mathtt{d}x_1+\,\cdots\,+x_n\,\mathtt{d}x_n)\,{d\mu_i}(\mathbf{x}).$$
Then the space of solutions to $M$ consists of all functions
$$\sum_{P \in {\rm Ass}(M)} \phi_P(\mathtt{d}x_1,\ldots,\mathtt{d}x_n).$$
A differential primary decomposition of $M$ is obtained from this
by reading $\mathtt{d}x_i$ as $\partial_{x_i}$. 
Indeed, the command {\tt differentialPrimaryDecomposition}
described in \cite{CC} is identical to
 our command {\tt solvePDE}. All examples
in \cite[Section 6]{CC} can be interpreted as solving PDE.

\section{Schemes and Coherent Sheaves}
\label{sec6}

The concepts of schemes and coherent sheaves are central to modern algebraic geometry.
These generalize varieties and vector bundles, and they encode
geometric structures with multiplicities. The point is that
the supports of coherent sheaves and other schemes are generally nonreduced.
 We here argue that our
linear PDE offer a useful way to think about the geometry of these objects.
 That perspective  motivated the writing of \cite[Section~3.3]{INLA}.
  
The affine schemes we consider are defined by ideals $I$ in a polynomial ring
$R $. Likewise, 
submodules $M$ of $R^k$ represent
coherent sheaves on $\CC^n$. We study the
affine scheme ${\rm Spec}(R/I)$ and the coherent sheaf given by
the module $R^k/M$. The underlying geometric objects are the  affine
varieties $V(I)$ and $V(M)$ in $\CC^n$. The latter was discussed in Section~\ref{sec3}.
The solution spaces ${\rm Sol}(I)$ or ${\rm Sol}(M)$ furnish
nonreduced structures on these varieties,
encoded in the integral representations due to Ehrenpreis--Palamodov.
According to Section \ref{sec4}, these are dual to differential primary decompositions.
Coherent 
sheaves were a classical tool in the analysis of linear PDE, 
but in the analytic category, where their role was largely
  theoretical.  The Ehrenpreis--Palamodov Fundamental Principle
appears in H\"ormander's book  under the header
 {\em Coherent analytic sheaves on Stein manifolds}
\cite[Chapter VII]{HORMANDER}. Likewise,
Treves' exposition, in the title of \cite[Section 3.2]{Treves},
 calls for {\em Analytic sheaves to the rescue}.
 By contrast, sheaves in this paper are 
 concrete and algebraic:  they are modules in {\tt Macaulay2}.

One purpose of this section is to explore how PDE and
their solutions behave under degenerations. We consider
ideals and modules whose generators depend on a parameter $\epsilon$.
This is modelled algebraically by working over 
the field $K = \CC(\epsilon)$ of
rational functions in the variable~$\epsilon$.
Algorithm \ref{alg:solvePDE}  can be applied to the polynomial
ring $R = K[x_1,\ldots,x_n]$ over that field.
We think of $\epsilon$ as a small quantity and we are interested in what
happens when~$\epsilon \rightarrow 0$.

Our discussion in this section is very informal. This is by design.
We present a sequence of examples that illustrates the geometric ideas.
The only formal result is Theorem \ref{thm:quot4}, which concerns
the role of the Quot scheme in parametrizing systems of linear PDE.

\begin{example}[$n=2$] \label{ex:n2}
Consider the prime ideal $I_\epsilon = \langle \partial_1^2 - \epsilon^2 \partial_2 \rangle$.
For nonzero parameters~$\epsilon $,
by  Theorem \ref{thm:Palamodov_Ehrenpreis}, the solutions 
to this PDE are represented as
one-dimensional integrals
 $$ \alpha_\epsilon(z_1,z_2) \, \,= \, \int {\rm exp} (\epsilon\, t\, z_1  \,+\, t^2 z_2) dt  \quad \in \,{\rm Sol}(I).  $$
By taking the limit for $\epsilon \rightarrow 0$, this yields arbitrary functions $a(z_2)$.
These are among the solutions to $I_0 = \langle\, \partial_1^2 \,\rangle $.
Other limit solutions are obtained via the reflection $t \mapsto -t$. Set
$$ \beta_\epsilon(z_1,z_2) \,\, = \, \int {\rm exp} (-\epsilon\, t\, z_1  \,+\, t^2 z_2) dt \quad \in \,{\rm Sol}(I).  $$
Note the similarity to the one-dimensional wave equation (\ref{eq:wave2})
with $c = \epsilon$. The solution for $\epsilon = 0$ is given
in (\ref{eq:wave3}). This is found from the integrals above by
taking the following limit:
\begin{equation}
\begin{aligned}
\label{eq:doubleline}
 {\rm lim}_{\epsilon \rightarrow 0} \,\frac{1}{2 \epsilon} \bigl( \alpha_\epsilon(z_1,z_2) - \beta_\epsilon(z_1,z_2)
 \bigr) \,& =  \int {\rm lim}_{\epsilon \rightarrow 0}\frac{{\rm exp}(\epsilon tz_1+t^2z_2)-{\rm exp}(-\epsilon tz_1+t^2z_2)}{2\epsilon}dt\\
 &=\, \int tz_1{\rm exp}(t^2z_2) dt\\
& = \, z_1\, b(z_2). 
 \end{aligned}
\end{equation}
We conclude that the general solution to $I_0$ equals
$\, \phi(z_1,z_2) \,\, = \,\, a(z_2) + z_1 b(z_2)$, where $b$ is
any function in one variable.
The calculus limit in (\ref{eq:doubleline}) realizes a
scheme-theoretic limit in the sense of algebraic geometry.
Namely, two lines in (\ref{eq:qc}) converge to a 
double line in $\CC^2$.

\end{example}

\begin{example}[$n=3$]
For $\epsilon \not= 0$ consider the curve 
 $t \mapsto (\epsilon t^3, t^4, \epsilon^2 t^2)$ in $\CC^3$.
Its prime ideal equals $I_\epsilon =  \langle \partial_1^2 - \partial_2 \partial_3, \partial_3^2 - \epsilon^4 \,\partial_2 \rangle$.
The solution space  ${\rm Sol}(I_\epsilon)$ consists of the functions 
\begin{equation}
\label{eq:phi3}
  \phi(z_1,z_2,z_3) \, = \, \int {\rm exp}( \epsilon t^3 z_1 + t^4 z_2 + \epsilon^2 t^2 z_3) dt.
 \end{equation}
What happens to these functions when $\epsilon $ tends to zero?
We address this question algebraically.
The scheme-theoretic limit of the given ideal $I_\epsilon$ is the ideal in Example~\ref{ex:312}.
This is verified by a Gr\"obner basis computation
(cf.~\cite[Section 15.8]{Eisenbud}).
Passing from ideals to their varieties, we see
a toric curve in $\CC^3$ that degenerates to a line with multiplicity four.

We claim that the formula  in (\ref{eq:niceintrep}) arises from
(\ref{eq:phi3}),  just as in Example \ref{ex:n2}. Namely,
set $i = \sqrt{-1}$ and let
$\phi_s \in {\rm Sol}(I_\epsilon)$ be the function that is obtained from $\phi$ 
in (\ref{eq:phi3}) by replacing  the parameter $t$ with $ i^s t$.
Then the following four functions on the left are solutions to $I_\epsilon$:
$$
\begin{matrix}
  \phi_0+\phi_1+\phi_2+\phi_3 & \longrightarrow & a(z_2),  \\
  \epsilon^{-1} (\phi_0+i \phi_1+i^2 \phi_2+i^3 \phi_3) & \longrightarrow & z_1 b(z_2),  \\
 \epsilon^{-2} (\phi_0+i^2 \phi_1+i^4 \phi_2+i^6 \phi_3) & \longrightarrow & z_1^2 c'(z_2) + 2 z_3 c(z_2), \\
 \epsilon^{-3} (\phi_0+i^3 \phi_1+i^6 \phi_2+i^9 \phi_3) & \longrightarrow & z_1^3 d'(z_2) + 6 z_1z_3 d(z_2).
\end{matrix}
$$
The functions obtained as limits on the right are precisely 
the four summands seen in (\ref{eq:niceintrep}).
Thus, the solution spaces to this family of PDE 
reflect the degeneration of the toric curve.
\end{example}

Such limits make sense also for modules.
 If a module $M_\epsilon \subseteq R^k$ 
depends on a parameter $\epsilon$ then we study
its solution space ${\rm Sol}(M_\epsilon)$ as $\epsilon $ tends to zero.
Geometrically, we examine flat families of coherent sheaves on $\CC^n$ or 
on $\PP^{n-1}$. A typical scenario comes from the action of the torus $(\CC^*)^n$, where 
Gr\"obner degenerations arise as limits under one-parameter subgroups.
The limit objects are monomial ideals (for $k=1$) or torus-fixed submodules (for $k \geq 2$).
The next example illustrates their rich structure with an explicit family of torus-fixed submodules.

\begin{example}[$n=2, k=3, l=6$] \label{ex:236}
Given a $3 \times 6$ matrix $A$ with random real entries, we~set
$$ M \,\,= \,\, {\rm image}_R \bigl(\, A \cdot {\rm diag}(\partial_1,\partial_1^2,\partial_1^3, \partial_2,\partial_2^2,\partial_2^3)\,\bigr) \quad \subset \quad R^3. $$
Then $M$ is torus-fixed and $\mathfrak{m}$-primary,
where $\mathfrak{m} = \langle \partial_1,\partial_2 \rangle$, and
${\rm amult}(M) = 10$. A basis of ${\rm Sol}(M)$
is given by ten polynomial solutions,
namely the standard basis vectors $e_1,e_2,e_3$, 
four vectors that are multiples of
$z_1,z_1,z_2, z_2$, and three vectors that are
multiples $z_1^2, z_1 z_2, z_2^2$. The reader is invited to
verify this with {\tt Macaulay2}.
Here is the input for one concrete instance:
\begin{verbatim}
R = QQ[x1,x2]
M = image matrix {{7*x1,5*x1^2,8*x1^3, 5*x2,9*x2^2,5*x2^3},
                  {8*x1,9*x1^2,8*x1^3, 4*x2,2*x2^2,4*x2^3},
                  {3*x1,2*x1^2,6*x1^3, 4*x2,4*x2^2,7*x2^3}}
solvePDE(M)
\end{verbatim}
By varying the matrix $A$, and by extracting 
the vector multipliers of  $1,z_1$ and $z_1^2$,
we obtain any complete flag of subspaces in $\CC^3$. 
The vector multipliers of $1$, $z_2$, and $z_2^2$ give us another complete flag of subspaces in $\mathbb{C}^3$, and the multiplier of $z_1z_2$ gives us the intersection line of the planes corresponding to the multipliers of $z_1$ and $z_2$.
This is illustrated in Figure~\ref{fig:flag}.
Thus {\em flag varieties}, with possible additional structure, appear naturally in such families.

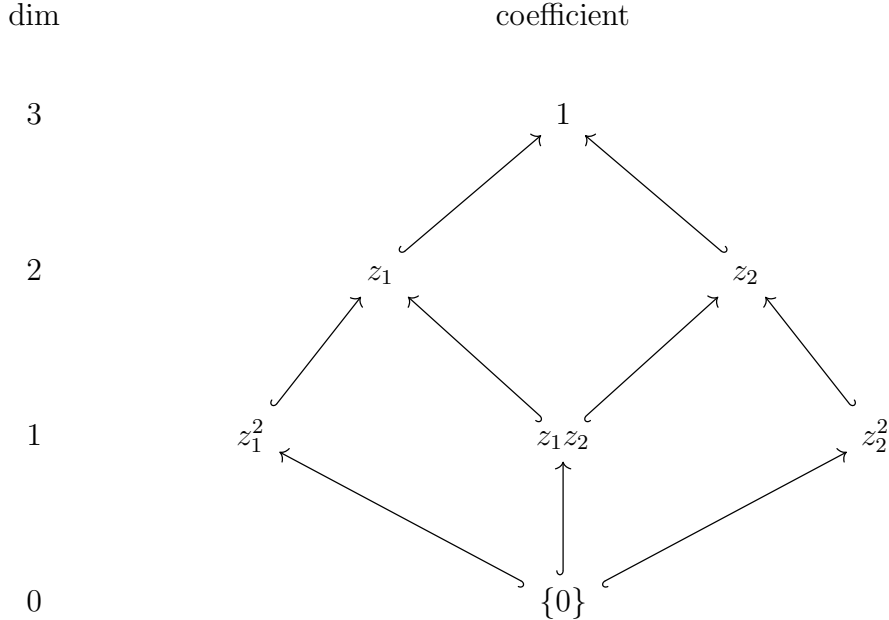
\begin{figure}[H]
  \centering
  \begin{tikzcd}
    {\text{dim}} &&&& {\text{coefficient}} \\
    3 &&&& 1 \\
    \\
    2 &&& {z_1} && {z_2} \\
    \\
    1 && {z_1^2} && {z_1z_2} && {z_2^2} \\
    \\
    0 &&&& {\{0\}}
    \arrow[hook, from=4-4, to=2-5]
    \arrow[hook, from=6-3, to=4-4]
    \arrow[hook, from=8-5, to=6-3]
    \arrow[hook', from=8-5, to=6-7]
    \arrow[hook', from=6-7, to=4-6]
    \arrow[hook', from=4-6, to=2-5]
    \arrow[hook, from=8-5, to=6-5]
    \arrow[hook, from=6-5, to=4-4]
    \arrow[hook', from=6-5, to=4-6]
  \end{tikzcd}
  \caption{The coefficient vectors of the solutions to the PDE in Example~\ref{ex:236} correspond to the above linear spaces with the given inclusions. We obtain two complete flags in $\mathbb{C}^3$, along with one interaction between the two.
  Experts on {\em quiver representations} will take note.
  \label{fig:flag}}
\end{figure}

\end{example}

The degenerations of ideals and modules we saw point us
to Hilbert schemes and Quot schemes.
Let us now also take a fresh look at Example \ref{ex:quotex1}.
The modules $M$ in that example form a flat family over the affine space $\CC^5$ with coordinates
${\bf c} = (c_1,c_2,c_3,c_4,c_5)$.
For ${\bf c} = 0$ we obtain the PDE whose solution space equals
$\CC \{e_1,z_1 e_1, z_1^2 e_1 \}$. But, what happens when 
one of the coordinates of ${\bf c}$ tends to infinity?
That limit exists in the Quot scheme.

In our context,
Hilbert schemes and Quot schemes serve as
parameter spaces for primary ideals and primary modules.
This was shown for ideals in \cite{CPS} and for modules in \cite{CC}.
In what follows we shall discuss the latter case.
Fix a prime ideal $P$ of codimension $c$ in $R = K[x_1,\ldots,x_n]$.
Write $\mathbb{K}$ for the field of fractions of the integral domain $R/P$,
as in Line~\ref{alg:line6} of Algorithm \ref{alg:solvePDE}.
We write $u_1,\ldots,u_n$ for the images in $\mathbb{K}$ 
of the variables $x_1,\ldots,x_n$ in~$R$.
After possibly permuting these variables, we shall assume that 
$P \cap K[x_{c+1},\ldots,x_n] = \{0\}$.  The set $\{u_{c+1},\ldots,u_n \}$ is 
algebraically independent over $K$,  so it serves as $\mathcal{S}$ in  Line \ref{alg:line5}.

Consider the formal power series ring $S =   \KK[[y_1,\ldots,y_c]]$
where $y_1,\ldots,y_c$ are new variables. This is a local ring
with maximal ideal $ \mathfrak{m} = \langle y_1,\ldots,y_c \rangle$.
We are interested in $\mathfrak{m}$-primary
submodules $L$ of $S^k$. The quotient module
 $S^k/L$ is finite-dimensional as a $\KK$-vector space, and we write
$\nu = {\rm dim}_\KK(S^k/L)$ for its dimension.  The {\em punctual Quot scheme}
is a parameter space whose points are precisely those modules. We denote the Quot scheme~by
\begin{equation}
\label{eq:quotscheme}
 {\rm Quot}^\nu (S^k) \,\,= \, \,
\bigl\{ \, L \subset S^k \,:\, L \,\,\hbox{submodule with}\,\, {\rm Ass}(L) = \mathfrak{m} 
\,\,\hbox{and}\,\, {\rm dim}_\KK(S^k/L) = \nu \,\bigr\}. 
\end{equation}
This is a quasi-projective scheme over $\KK$, i.e.~it can be defined
by a finite system of polynomial equations and inequations in a large but finite
set of variables. Each solution to that system is one submodule $L$. This construction
goes back to Grothendieck, and it plays a fundamental role in parametrizing
coherent sheaves in algebraic geometry. While a constructive approach to Quot schemes
exists, thanks to Skjelnes \cite{Skjelnes}, the problem remains to write 
defining equations  for ${\rm Quot}^\nu (S^k) $   in a computer-readable format, for small values of $c, k, \nu$.
A natural place to start would be the case $c=2$, given that coherent sheaves
supported at a smooth point on a surface are of considerable interest in
geometry and physics \cite{AJLOP, Bar, ellingsrud1999irreducibility, HJM}.

The next two examples offer a concrete illustration of the concept of Quot schemes.
We exhibit the Quot schemes that parametrize two families of linear PDE
we encountered before.

\begin{example} [$c=2,k=3,\nu=10$] \label{ex:23ten}
Consider the formal power series ring $S = \KK[[y_1,y_2]]$
where $\KK$ is any field. Replacing $\partial_1,\partial_2$
with $y_1,y_2$ in Example \ref{ex:236}, every $ 3 \times 6$ matrix $A$
over $\KK$ defines a submodule $L$ of $S^3$.
The quotient $S^3/L$ is a $10$-dimensional $\KK$-vector space,
so $L$ corresponds to a point in the Quot scheme $  {\rm Quot}^{10} (S^3) $.
By varying $A$, we obtain a closed subscheme of 
$  {\rm Quot}^{10} (S^3) $, which contains the complete flag variety
we saw in  Example \ref{ex:236}.
\end{example}

\begin{example}   \label{ex:quotex3}
For $S = \KK[[y_1,y_2]]$, the scheme $ {\rm Quot}^{\nu} (S^k) $
is an irreducible variety of dimension $k \nu - 1$, 
by \cite[Theorem 2.2]{Bar}.
If $k=2, \nu = 3$ then this dimension is five.
The affine space with coordinates
${\bf c}$ in Example \ref{ex:quotex1} is
a dense open subset $W$ of ${\rm Quot}^3(S^2)$, by \cite[Section 7]{Bar}.
\end{example}

For $k=1$, the Quot scheme
is the {\em punctual Hilbert scheme} ${\rm Hilb}^\nu(S)$;
see   \cite{Bri77}. The points on this
Hilbert scheme represent $\mathfrak{m}$-primary ideals of length $\nu$ 
in $S =   \KK[[y_1,\ldots,y_c]]$.
It was shown in \cite[Theorem 2.1]{CPS} that ${\rm Hilb}^\nu(S)$ parametrizes the
set of all $P$-primary ideals in $R$ of multiplicity $\nu$. This means that we can
encode $P$-primary ideals in $R$ by $\mathfrak{m}$-primary ideals in $S$,
thus reducing scheme structures on any higher-dimensional variety to a
scheme structure on a single point.  This was generalized from 
ideals to submodules ($k \geq 2$)
by Chen and Cid-Ruiz  \cite{CC}. Geometrically, we encode
coherent sheaves by those supported at one point, namely the
generic point of $V(P)$, corresponding to the field extension $\mathbb{K}/K$.
Here is the main result from \cite{CC}, 
stated for the polynomial ring $R$,  analogously to \cite[Theorem~2.1]{CPS}.

\begin{theorem} \label{thm:quot4}
	The following four sets of objects are in a natural bijective correspondence: \vspace{-0.15cm}
	\begin{enumerate}[(a)]
		\item $P$-primary submodules $M$ in $R^k$ of multiplicity $\nu$ over $P$, \vspace{-0.15cm}
    \item $\mathbb{K}$-points in the punctual Quot scheme $\,{\rm Quot}^\nu \! \left(\,\KK[[y_1,\ldots,y_c]]^k\right)$,
		\vspace{-0.15cm}
		\item $\nu$-dimensional $\KK$-subspaces of $\,\KK[z_1,\ldots,z_c]^k$ that are closed under differentiation, \vspace{-0.15cm}
		\item  $\nu$-dimensional $\KK$-subspaces of
		the Weyl-Noether module $ \KK \,\otimes_R  \,D_{n,c}^k$ that are $R$-bimodules. 
	\end{enumerate}
	Moreover, any basis of the $\KK$-subspace (d)  can be lifted to a finite subset $\mathcal{A} $ of $ D_{n,c}^k$
	such that
\begin{equation}
\label{eq:modulemembership2}
M \,\,\, = \,\,\, \bigl\{ \,m \in R^k \,: \, \delta \bullet m \in P
\,\,\,\hbox{for all}\,\,\, \delta \in \mathcal{A}   \,\bigr\}. 
\end{equation}
\end{theorem}

Here  $D_{n,c}$ is the subalgebra of the Weyl algebra $D_n$ consisting of all
operators (\ref{eq:Ann}) with  $s_{c+1} = \cdots = s_n = 0$. 
 This is a special case of (\ref{eq:usingonly}).  Elements 
in $D_{n,c}$ are differential operators in $\partial_{x_1},\ldots,\partial_{x_c}$
whose coefficients are polynomials in $x_1,\ldots,x_n$. Note that
$D_{n,0} = R$ and $D_{n,n} = D_n$.
 Equation (\ref{eq:modulemembership2}) says that $\mathcal{A} $ is
 a differential primary decomposition for the primary module $M$.
 The Noetherian operators in
  $\mathcal{A}$ characterize membership in $M$.
 In this paper, however, we focus on the   $\mathbb{K}$-linear
 subspaces in item (c). By clearing denominators, we can
 represent such a subspace by a basis  $\mathcal{B}$ of
 elements in $K[x_1,\ldots,x_n][z_1,\ldots,z_c]$.
 These are precisely the Noetherian multipliers needed for
 the integral representation of ${\rm Sol}(M)$.
 In summary, Theorem \ref{thm:quot4} may be understood 
 as a theoretical counterpart to Algorithm \ref{alg:solvePDE}.
The following example elucidates the important role 
played by the Quot scheme in our algorithm.

\begin{example}[$n=4,c=2,k=3,\nu=10$]
Let $P$ be the prime ideal in \cite[equation (1)]{CPS}.
Equivalently, $P$ is the prime $P_6$ in Example \ref{ex:mod1a}.
The surface $V(P) \subset \CC^4$ is the cone over the twisted cubic curve.
Consider the point in ${\rm Quot}^{10}(\KK[[y_1,y_2]]^3)$
given by a matrix $A$ as in Example \ref{ex:23ten}.
The bijection from  (b) to (a) in Theorem  \ref{thm:quot4} yields
a $P$-primary submodule $M$ of multiplicity $10$ in $K[x_1,\ldots,x_4]^3$.
Generators for the module $M$ are found by computing the inverse image under
the map $\gamma$, as shown in \cite[equation (2)]{CC}.
This step is the analogue for modules of the  elimination
that creates a  large $P$-primary ideal $Q$ from
\cite[equation~(5)]{CPS}.
Geometrically speaking, the $10$-dimensional space of polynomial vectors that 
are solutions to the PDE in  Example \ref{ex:236} encodes a 
 coherent sheaf of rank $3$ on the singular surface $V(P)$.
\end{example}

The ground field $K$ in Section~\ref{sec5} need not be algebraically closed.
In particular, we usually take $K=\mathbb{Q} $ when computing in {\tt Macaulay2}.
But  this requires some adjustments in our results. For instance,
Theorem \ref{thm:finitedim} does not apply when the coordinates of ${\bf u} \in \CC^n$
are not in $K$. In such situations, we may take $\KK$ to be an algebraic extension of $K$.
We close with an example that shows the effect of the choice of ground field in a concrete computation.

\begin{example}[$n=k=2,l=3$] \label{ex:67}
Consider the module $M$ given in {\tt Macaulay2} as follows:
\begin{verbatim}
R = QQ[x1,x2];  M = image matrix {{x1,x1*x2,x2},{x2,x1,x1*x2}};
                    dim(R^2/M), degree(R^2/M), amult(M)
\end{verbatim}
The output shows that ${\rm amult}(M) = 5$ when $K = \mathbb{Q}$,
but $\nu={\rm amult}(M) = 6$ when $K = \CC$:
Applying now the command {\tt solvePDE(M)}, we find the
differential primary decomposition
\begin{small}
\begin{verbatim}
{{ideal (x2, x1), {| 1 |, | 0 |, | -dx1 |}}, {ideal (x2 - 1, x1 - 1), {| -1 |}},               
                   | 0 |  | 1 |  |  dx2 |                              |  1 |                          
                        2            
 {ideal (x1 + x2 + 1, x2  + x2 + 1), {| x2+1 |}}}
                                      |   1  |
\end{verbatim}
\end{small}
The module $M$ has three associated primes over $K=\QQ$.
The first gives three polynomial solutions, including
$\binom{-z_1\,}{\phantom{-}z_2\,}$.
The second prime contributes $\binom{-1\,}{\phantom{-}1\,}{\rm exp}(z_1+z_2)$,
and the third gives $\binom{x_2+1}{1}{\rm exp}(x_1 z_1 + x_2 z_2)$,
where $(x_1,x_2)$ is 
$\frac{1}{2}(-1+\sqrt{3}i,-1-\sqrt{3}i)$ or
$\frac{1}{2}(-1-\sqrt{3}i,-1+\sqrt{3}i)$.
Here $\,\KK = \mathbb{Q}(\sqrt{3}i)\,$ is the field extension of 
$\,K = \mathbb{Q}\,$ defined 
by the third associated prime.
\end{example}

\section{What Next?}
\label{sec7}

The results presented in this article suggest
many directions for future study and research.

\subsection{Special Ideals and Modules}
\label{subsec71}

One immediate next step is to explore the
PDE corresponding to various specific
ideals and modules that have appeared in the
literature in commutative algebra and algebraic geometry.

One interesting example is the class of ideals studied 
recently by Conca and Tsakiris in \cite{CT},
namely products of linear ideals. 
A minimal primary decomposition for such an ideal $I$
is given in \cite[Theorem 3.2]{CT}. 
It would be gratifying to find the arithmetic multiplicity  ${\rm amult}(I)$
and the solution spaces ${\rm Sol}(I)$ in terms of matroidal data
for the subspaces in $V(I)$.

A more challenging problem is to compute the solution space ${\rm Sol}(J)$
when $J$ is an ideal generated by $n$ power sums in $R = \QQ[\partial_1,\ldots,\partial_n]$.
This problem is nontrivial even for $n=3$. To be more precise, we fix 
relatively prime integers $0 < a < b< c$, and 
we consider the ideal
$$ J_{a,b,c} \,\,=\,\, \langle\,
 \partial_1^a +   \partial_2^a +   \partial_3^a \,,\,
 \partial_1^b +   \partial_2^b +   \partial_3^b\, ,\,
 \partial_1^c +   \partial_2^c +   \partial_3^c \,\rangle.  $$
If $(a,b,c) = (1,2,3) $ then $V(J_{1,2,3}) = \{0\}$
and ${\rm Sol}(J_{1,2,3})$ is a six-dimensional
space of polynomials, spanned by the discriminant
$(z_1-z_2)(z_1-z_3)(z_2-z_3)$ and its successive derivatives.
In general, it is conjectured in \cite{CKW}
that $V(J_{a,b,c}) = \{0\}$
if and only if $abc$ is a multiple of $6$.
If this holds then ${\rm Sol}(J_{a,b,c})$ 
consists of polynomials. If this does not hold
then we must compute  $V(J_{a,b,c})$
and extract the Noetherian multipliers from all
associated primes of $J_{a,b,c}$.
For instance, for $(a,b,c) = (2,5,8)$ with $K=\QQ$,
the arithmetic multiplicity is $38$,
one associated prime is $\langle \partial_1+\partial_2+\partial_3,
\partial_2^2 + \partial_2 \partial_3 + \partial_3^2 \rangle$,
and the largest degree of a polynomial solution is $10$.

It will be worthwhile to explore the solution spaces
${\rm Sol}(M)$ for modules $M$ with special combinatorial structure.
One natural place to start are syzygy modules. For instance, take
\begin{equation}
\begin{small}
\label{eq:koszul} A(\partial) \quad = \quad \begin{pmatrix}
\, \partial_2 & \partial_3 & \partial_4 &    0   &    0 &   0 \\
- \partial_1 &      0      &     0         & \partial_3 & \partial_4 &  0 \\
\,       0          & \! - \partial_1 & 0         & \! - \partial_2 & 0     & \partial_4 \\
\,       0         &      0            & \! - \partial_1 & 0    & \!\! -\partial_2 & \! \! - \partial_3\,\,
\end{pmatrix},
\end{small}
\end{equation}
which is the first matrix in the Koszul complex for $n=4$.
Since ${\rm rank}(A) = 3$,  the module $M = {\rm image}_R(A)$ 
is supported on the entire space, i.e.~$V(M)=\CC^4$. Its solutions are the
gradient vectors $\nabla \alpha = \sum_{j=1}^4 \frac{\partial \alpha}{\partial z_j} e_j$, 
where $\alpha = \alpha(z_1,z_2,z_3,z_4)$ ranges over all
functions in $\mathcal{F}$.

Toric geometry \cite[Chapter 8]{INLA}
furnishes modules whose PDE should be interesting.
The initial ideals of a toric ideal with respect to weight vectors
are binomial ideals, so the theory of binomial primary decomposition
applies, and it gives regular polyhedral subdivisions as in 
\cite[Theorem 13.28]{INLA}.
Non-monomial initial ideals should be studied from the differential point of view.
Passing to coherent sheaves, we may examine the modules representing
toric vector bundles and their Gr\"obner degenerations.
In particular, the cotangent bundle of an embedded toric variety, in
affine or projective space, is likely to encompass intriguing PDE.

\subsection{Linear PDE with Polynomial Coefficients}
\label{subsec72}

We discuss an application 
to PDE with non-constant coefficients,
here taken to be polynomials.
Our setting is the  Weyl~algebra
$ D \,\,= \,\, \CC \langle z_1,\ldots,z_n, \partial_1,\ldots,\partial_n \rangle$.
A linear system of PDE with polynomial coefficients is a {\em $D$-module}.
For instance, consider a $D$-ideal $I$, that is, a left
ideal in the Weyl algebra $D$.
The solution space of $I$ is typically infinite-dimensional.

We construct solutions to $I$ with the
method of {\em Gr\"obner deformations} 
\cite[Chapter 2]{SST}. Let $w \in \RR^n$ be a general weight vector,
and consider the initial $D$-ideal ${\rm in}_{(-w,w)}(I)$. This is also a $D$-ideal,
and it plays the role of Gr\"obner bases in solving polynomial equations.
We know from \cite[Theorem 2.3.3]{SST} that 
${\rm in}_{(-w,w)}(I)$ is fixed under the natural action of the
$n$-dimensional algebraic torus $(\CC^*)^n$ on the Weyl algebra $D$.
This action is given in \cite[equation (2.14)]{SST}.
It gives rise to a Lie algebra action generated by the $n$ Euler operators
$$ \theta_i \, = z_i \partial_i \quad {\rm  for} \,\,\, i = 1,2,\ldots,n. $$
These Euler operators commute  pairwise, and they generate a (commutative) polynomial subring
$\CC[\theta] = \CC[\theta_1,\ldots,\theta_n]$ of the Weyl algebra $D$.
If $J$ is any torus-fixed  $D$-ideal then it is generated by
operators of the form $x^a p(\theta) \partial^b$ where $a,b \in \mathbb{N}^n$.
We define the falling factorial 
$$ [\theta_b] \,\, := \,\, \prod_{i=1}^n \prod_{j=0}^{b_i - 1} (\theta_i - j). $$
The {\em distraction}  $ \widetilde J $ is the ideal in $\CC[\theta]$
generated by all polynomials 
$\, [\theta_b] p(\theta-b) \, = \, x^b p(\theta) \partial^b $,
where $ x^a p(\theta) \partial^b$ runs over a generating set of $J$.
The space of classical solutions to $J$ is equal to that of $\widetilde J$.
This was exploited in \cite[Theorem 2.3.11]{SST} under the 
assumption that $J$ is holonomic, which means that $\widetilde J$ is zero-dimensional
in $\CC[\theta]$. We here drop that assumption.

Given any $D$-ideal $I$, we compute its
initial $D$-ideal $J = {\rm in}_{(-w,w)}(I)$ for $w \in \RR^n$ generic.
Solutions to $I$ degenerate to solutions of $J$
under the Gr\"obner degeneration given by $w$. 
We can often reverse that construction:
given solutions to $J$, we lift them
to solutions of $I$. 
Now, to construct all solutions of $J$ we study the
{\em Frobenius ideal} $F=\widetilde J$. This is an ideal in  $\CC[\theta]$.

We now describe all solutions to a given ideal
$F$ in $\CC[\theta]$. This was done in \cite[Theorem 2.3.11]{SST}
for zero-dimensional $F$.
 Ehrenpreis--Palamodov allows us to solve the general case.
Here is our algorithm.
We replace each operator $\theta_i = z_i \partial_i$ by the corresponding
$\partial_i$. We then apply {\tt solvePDE} to get
the general solution to the linear PDE with constant coefficients.
In that general solution, we now replace each coordinate $z_i$
by its logarithm ${\rm log}(z_i)$. In particular, each occurrence
of ${\rm exp}(u_1z_1+ \cdots +u_n z_n)$ is replaced by a formal monomial
$z_1^{u_1}  \cdots z_n^{u_n}$. The resulting expression represents
the general solution to the Frobenius ideal~$F$.

\begin{example}
As a warm-up, we note that a function in one
variable $z_2 $ is annihilated by the squared Euler operator $\,\theta_2^2 = 
z_2 \partial_2 z_2 \partial_2\,$ if and only if
it is a $\CC$-linear combination of $1$ and ${\rm log}(z_2)$.
Consider the Frobenius ideal given by Palamodov's
system \cite[Example 11]{CPS}:
$$ F \,\,= \,\, \langle \,\theta_2^2\,, \,\theta_3^2\,,\, \theta_2 - \theta_1 \theta_3 \,\rangle .$$
To find all solutions to $F$, we consider the corresponding ideal 
$ \langle \,\partial_2^2\,, \,\partial_3^2\,,\, \partial_2 - \partial_1 \partial_3 \,\rangle $
in $\CC[\partial_1,\partial_2,\partial_3]$.
By {\tt solvePDE}, the general solution to that constant coefficient system equals
$$ \alpha(z_1) \,\,+\,\, z_2 \cdot \beta'(z_1) \,\,+ \,\,z_3 \cdot \beta(z_1), $$
where $\alpha$ and $\beta$ are functions in one variable.
We now replace $z_i$ by ${\rm log}(z_i)$ and we abbreviate
$A(z_1) = \alpha({\rm log}(z_1))$ and
$B(z_1) = \beta({\rm log}(z_1))$. Thus $A$ and $B$ are again arbitrary functions in one variable.
We conclude that the general solution to the given Frobenius ideal $F$ equals
$$ \phi(z_1,z_2,z_3) \,\,\, = \,\,\,A(z_1) \,+\, z_1 \cdot {\rm log}(z_2) \cdot B'(z_1) \,+\, {\rm log}(z_3) \cdot B(z_1). $$
 This method can also be applied for $k \geq 2$, enabling us to study
 solutions for any $D$-module.
\end{example}

\subsection{Socle Solutions}

The solution space ${\rm Sol}(M)$ to a system $M$ of linear PDE 
is a complex vector space, typically infinite-dimensional.
Our algorithm in Section \ref{sec5} decomposes that  space into 
finitely many natural pieces, one for each of the integrals 
in (\ref{eq:anysolution}). 
The number ${\rm amult}(M)$ of pieces is a meaningful
invariant from commutative algebra.
Each piece is labeled by a polynomial
$B_{ij}({\bf x},{\bf z})$ in $2n$ variables, and it is
parametrized by measures $\mu_{ij}$ on the irreducible variety~$V_i$.

This approach does not take full advantage of the 
fact that ${\rm Sol}(M)$ is an $R$-module 
where $R = \CC[\partial_1,\ldots,\partial_n]$.
Indeed, if $\psi({\bf z})$ is any solution to $M$ then so
is $(\partial_i \bullet \psi)({\bf z})$. So, if we list all solutions then
$\partial_i \bullet \psi$ is redundant provided $\psi$ is already listed.
More precisely, we consider
\begin{equation} \label{eq:socle}
 {\rm Sol}(M) / \langle \partial_1,\ldots,\partial_n \rangle {\rm Sol}(M). 
 \end{equation}
This quotient space is still infinite-dimensional over $\CC$, but
it often has a much smaller description than
${\rm Sol}(M)$.
A solution to $M$ is called a {\em socle solution} if it is nonzero in
(\ref{eq:socle}).
We pose the problem of modifying {\tt solvePDE}
so that the output is a minimal subset of
Noetherian multipliers which represent all the
socle solutions. The solution will require the prior development
of additional theory in commutative algebra, along the lines of \cite{CC, CPS, CS}.

The situation is straightforward in the
case of Theorem \ref{thm:finitedim} when
the support $V(M)$ is finite.
Here the space ${\rm Sol}(M)$ is finite-dimensional, and it is
canonically isomorphic to the vector space dual
of $R^k/M$, as shown in \cite{oberst96}. Finding the socle
solutions is a computation using linear algebra
over $K = \CC$, similar to the three steps after
Proposition \ref{prop:poly_sol_deg_bound}. For instance,
let $k=1$ and suppose that $I$ is
a homogeneous ideal in $R$.
The socle solutions are sometimes called {\em volume polynomials}
\cite[Lemma 3.6.20]{SST}. The most desirable case
arises when $I$ is {\em Gorenstein}. Here the socle solution is unique
up to scaling, and it fully characterizes $I$.
For instance, consider the power sum ideal
$\,\langle \,\sum_{i=1}^n \partial_i^s \, : \,s=1,\ldots,n \,\rangle$.
This is Gorenstein with volume polynomial
$\Delta = \prod_{1 \leq i < j \leq n} (z_i-z_j)$.
For $n=3$, the ideal $I$ is $J_{1,2,3}$ in Subsection~\ref{subsec71}.
Here ${\rm Sol}(I)$ is a $\CC$-vector space of 
dimension $n!$. However, as an $R$-module, it is generated by a single
polynomial $\Delta$. A future version of {\tt solvePDE} should simply output
$\,{\rm Sol}(I) = R \Delta$.

It is instructive to revisit the general solutions to PDE we 
presented in this paper, and to highlight the socle solutions for each of them.
For instance, in Example~\ref{ex:312} we have ${\rm amult}(I) = 4$ 
but only one of the four Noetherian multipliers $B_i$ 
gives a  socle solution. The last summand 
in $(\ref{eq:niceintrep})$ gives the socle solutions.
The first three summands can be obtained from the last summand
by taking derivatives. What are the socle solutions in Example
\ref{ex:mod1a}?

\subsection{From Calculus To Analysis}

The storyline of this paper is meant to be accessible
for students of multivariable calculus.
These students know how to check that (\ref{eq:pdesys2})
is a solution to (\ref{eq:pdesys1}).
The derivations in
Examples \ref{ex:312}, \ref{ex:mod1a},
\ref{ex:51},  \ref{ex:53},
\ref{ex:n2}, \ref{ex:236}   and \ref{ex:67}
are understandable as well.
No prior exposure to abstract algebra is needed to follow these
examples, or to download
{\tt Macaulay2} and run   {\tt solvePDE}.

The objective of this subsection is to move beyond calculus, and to
build a bridge to advanced themes and current research 
in analysis. First of all, we ought to consider
inhomogeneous systems of linear PDE
with constant coefficients. Such a system has the~form
\begin{equation}
\label{eq:inhomogeneous}
\,A(\partial) \bullet \psi({\bf z}) \,=\, f({\bf z}) ,
\end{equation} where $A$ is
a $k \times l$ matrix as before
and $f$ is a vector in $ \mathcal{F}^l$,
where $\mathcal{F}$ is a  space of
functions or distributions. 
Writing $a_i$ for the $i$th column of $A$, the
system (\ref{eq:inhomogeneous}) describes vectors $\psi = (\psi_1,\ldots,\psi_k)$ 
with $a_i \bullet \psi = f_i$ for $i=1,\ldots,l$.
The study of the inhomogeneous system (\ref{eq:inhomogeneous})
is a major application of 
Theorem \ref{thm:Palamodov_Ehrenpreis}.
We see this in Palamodov's book \cite[Chapter VII]{PALAMODOV}, but also in 
the work of Oberst who addresses the
``canonical Cauchy problem'' in \cite[Section 5]{oberst90}.
An important role is played by the syzygy module ${\rm ker}_R(A) \subset R^l$,
whose elements are the $R$-linear relations on the columns $a_1,\ldots,a_l$.
A necessary condition for solvability of (\ref{eq:inhomogeneous})
is that the Fourier transform of the right hand side $f = (f_1,\ldots,f_l)$
satisfies the same syzygies. H\"ormander shows 
in \cite[Theorem 7.6.13]{HORMANDER} that the converse is also true,
under certain regularity hypotheses on $f$. Thus the computation of
syzygies and other homological methods (cf.~\cite[Part III]{Eisenbud})
are useful for solving (\ref{eq:inhomogeneous}).
Treves calls this {\em Simple algebra in the general case} \cite[Section 3.1]{Treves}.
We point to his exact sequence in \cite[equation (3.5)]{Treves}.
 Syzygies can be lifted to D-modules \cite[Section 2.4]{SST}
 via the Gr\"obner deformations in
 Subsection~\ref{subsec72}.

Another issue is to better understand which
collections of vectors $B_{ij}$ arise as
Noetherian multipliers for some PDE.
The analogous question for
Noetherian operators of ideals is addressed in \cite[Theorem 3.1]{CPS}.
That result is essentially equivalent to the characterization 
in \cite[Theorem 7.7.7]{HORMANDER} of spaces 
   $\mathcal{A}$ of Noetherian operators
for a primary module as being closed under the Lie bracket.
More work on this topic is needed.
This is related to the
issue of primary fusion, discussed at the end of
\cite[Section 5]{CS}, which concerns the compatibility of
minimal sets of Noetherian operators for 
associated primes that are contained in one another.

We end with a pointer to current research in calculus of variations
by De Phillippis and collaborators in~\cite{ADHR, DR}.
Each solution $\mu$ to the PDE $A \bullet \mu = 0$
is a Radon measure on an open set
in $\RR^n$ with values in $\RR^k$.  Such a measure $\mu$
is called {\em $A$-free}, and one is interested in the
singular part $\mu^s $ of $\mu$. Analysts
view solutions among smooth functions as classical and well-understood,
and they care primarily about irregularities and their rectifiability.
One studies $\mu^s$ via the  polar vector
$\frac{{\rm d} \mu}{{\rm d} |\mu|}$ in $\RR^k$.
The main result in \cite{DR} states that this vector always
lies in the wave cone $\Lambda_A$.  This 
is a real algebraic variety in $\RR^k$ which is an invariant
of our module $M = {\rm image}_R(A)$. 
When $A = {\rm curl} \,$ as in  (\ref{eq:koszul}),
the wave cone is a Veronese  variety, and the result is
Alberti's Rank-One Theorem.
The article \cite{ADHR} proves  the same conclusion  for more refined 
 wave cones, and it offers a conjecture relating
the geometry of wave cones to the singular supports
of solutions \cite[Conjecture 1.6]{ADHR}.
It would be interesting to compute these real varieties
in practice, and to learn about $A$-free measures
from the output of {\tt solvePDE}.

\subsection{Numerical Algebraic Geometry}

In applications, one often does not have access to an exact representation of a problem, but rather some approximation with possible errors introduced by measurements or finite-precision arithmetic.
The last decade of developments in \emph{numerical algebraic geometry} \cite{BHSW} provides tools for the numerical treatment of such polynomial models.
In that paradigm, a prime ideal $P \subset \CC[\mathbf{x}]$ 
is represented by a \emph{witness set}, i.e. a set of $\operatorname{deg}(P)$ points approximately on $V(P) \cap L$, where $L$ is a generic affine-linear 
space of dimension $c = \operatorname{codim}(P)$.
Similarly, radical ideals are  collections of witness sets corresponding to irreducible components.
Dealing with general ideals and modules is much more subtle,
since these have embedded primes.
One idea, pioneered by Leykin \cite{LEYKIN}, is to consider \emph{deflations} of ideals.
Modules were not considered in \cite{LEYKIN}.
Deflation has the effect of exposing embedded and non-reduced components as isolated components, which can subsequently be represented using witness sets.
One drawback is that the deflated ideal lies in a polynomial ring with many
new variables. 

We advocate the systematic development of numerical methods
for linear PDE with constant coefficients. Noetherian operators and
multipliers can be used to  represent arbitrary ideals and modules.
For each prime $P$,  both the field 
$\mathbb{K}=\operatorname{Frac}(R/P)$
and the spaces in (\ref{eq:frakAB})
should be represented purely numerically.
Along the way, one would extend the current repertoire of
numerical algebraic geometry to modules
and their coherent sheaves.

First steps towards the numerical encoding of affine schemes
were taken in \cite{CHKL}, for ideals $I$ with
 no embedded primes. 
 The key observation is that
the coefficients of the Noetherian operators for the $P$-primary component of $I$  can be evaluated at a point
 ${\bf u} \in V(P)$ using only linear algebra over $\CC$.
 This linear algebra step can be carried out purely numerically.
 
Inspired by this, we propose a numerical representation of an arbitrary module
 $M \subseteq R^k$.
Let $(P_i, \mathcal{S}_i, \mathcal{A}_i)$ be a differential primary decomposition 
as in  Theorem \ref{thm:DPD}.
Assuming the ability to sample generic points $\mathbf{u}_i \in V(P_i)$,
we encode the sets $\mathcal{A}_i$ by their point evaluations $\mathcal{A}_i(\mathbf{u}_i) = \{ A(\mathbf{u}_i, \mathbf{\partial_\mathbf{x}}) \colon A(\mathbf{x}, \partial_\mathbf{x}) \in \mathcal{A}_i \}$.
Each evaluated  operator $ A(\mathbf{u}_i, \mathbf{\partial_\mathbf{x}})$
 gives an exponential solution 
$B(\mathbf{u}_i, \mathbf{z}) \exp(\mathbf{u}_i^t \,\mathbf{z})$ 
to the PDE given by $M$ via the correspondence in Proposition~\ref{prop:NONM}.
We obtain a numerical module membership test: a polynomial
vector $m \in R^k$ belongs to $M$ with high probability if $A(\mathbf{u}_i, \partial_\mathbf{x}) \bullet m$ vanishes at the point $\mathbf{u}_i$ for all $A \in \mathcal{A}_i(\mathbf{u}_i)$ and $i = 1,\dotsc,s$.  The exponential functions
${\bf z} \to B(\mathbf{u}_i, \mathbf{z}) \exp(\mathbf{u}_i^t \,\mathbf{z})$,
which depend on numerical parameters $\mathbf{u}_i$,
serve as an encoding of
the infinite-dimensional $\CC$-vector space ${\rm Sol}(M)$.

Another potential research direction is the development of hybrid
 algorithms, where numerical information is used to speed up symbolic computations.
Assuming the numerical approximations to be accurate enough, the output of a 
hybrid algorithm is exact.
For Noetherian operators of ideals with no embedded components, this is explored in \cite{CHKL}, and it is already 
implemented in the {\tt Macaulay2} package \texttt{NoetherianOperators} \cite{CCHKL} 
using the command \texttt{noetherianOperators(I, Strategy => "Hybrid")}.
It will be desirable to extend this hybrid method to the command \texttt{solvePDE},
in the full generality seen in Algorithm  \ref{alg:solvePDE}.

In conclusion, the numerical solution of partial differential equations
is the key to computational science. The case of linear PDE
with constant coefficients serves as a base case. 
We hope that the  techniques described in this article
will be useful for the future applications.

\bigskip \bigskip

\noindent
\footnotesize 
{\bf Authors' addresses:}

\smallskip

\noindent Rida Ait El Manssour, MPI-MiS Leipzig
\hfill {\tt rida.manssour@mis.mpg.de}

\noindent Marc H\"ark\"onen, Georgia Institute of Technology
\hfill {\tt harkonen@gatech.edu}

\noindent Bernd Sturmfels,
MPI-MiS Leipzig and UC Berkeley 
\hfill {\tt bernd@mis.mpg.de}
\end{document}